\documentclass[a4paper,11pt]{amsart}

\usepackage{latexsym}
\usepackage{amssymb,amsmath}
\usepackage{graphicx}
\usepackage{url}
\usepackage[vcentermath]{youngtab}
  
\usepackage{booktabs}

\newtheorem{theorem}{Theorem}[section]
\newtheorem*{theorem*}{Theorem}
\newtheorem{corollary}[theorem]{Corollary}

\newtheorem{proposition}[theorem]{Proposition}

\newtheorem{lemma}[theorem]{Lemma}

\theoremstyle{definition}
\newtheorem*{definition}{Definition}

\newcommand{\N}{\mathbf{N}}

\newcommand{\Q}{\mathbf{Q}}
\newcommand{\R}{\mathbf{R}}

\renewcommand{\O}{\mathrm{O}}
\newcommand{\e}{\mathrm{e}}

\renewcommand{\epsilon}{\varepsilon}

\DeclareMathOperator{\Sym}{Sym}
\DeclareMathOperator{\BSym}{BSym}
\DeclareMathOperator{\DSym}{DSym}

\DeclareMathOperator{\Tr}{Tr}
\DeclareMathOperator{\sgn}{sgn}
\DeclareMathOperator{\Inf}{Inf}

\newcommand{\Ind}{\big\uparrow}
\newcommand{\Res}{\big\downarrow}
\newcommand{\ind}{\!\uparrow}
\newcommand{\res}{\!\downarrow}

\newcommand{\mfrac}[2]{{\textstyle\frac{#1}{#2}}}

\newcounter{thmlistcnt}
	{\setcounter{thmlistcnt}{0}%
	\begin{list}{\emph{(\roman{thmlistcnt})}}{%
		\usecounter{thmlistcnt}%
		\setlength{\topsep}{0pt}%
		\setlength{\leftmargin}{0pt}%
		\setlength{\itemsep}{0pt}%
		\setlength{\labelwidth}{17pt}
		\setlength{\itemindent}{30pt}}%
	}%
	{\end{list}}%

\newcommand{\stir}[2]{\genfrac{\{}{\}}{0pt}{}{#1}{#2}}
\newcommand{\id}{\mathrm{id}}
\newcommand{\Cl}{\mathrm{Cl}}
\newcommand{\D}{\mathrm{Des}}
\renewcommand{\theta}{\vartheta}
\newcommand{\Fix}{\mathrm{Fix}}

\newcommand{\sh}{\mathrm{sh}}

\newcommand{\BB}{B^{\dagger}} %was \mathbf{B}
\newcommand{\BBp}{B^{\dagger\prime}} %was \mathbf{B}'
\newcommand{\SB}{S^\dagger} %was \mathbf{S}
\newcommand{\SBp}{S^{\dagger\prime}} %was \mathbf{S}'
\newcommand{\MB}{M^\dagger} %was \mathbf{M}
\newcommand{\MBp}{M^{\dagger\prime}} %was \mathbf{M}

\newcommand{\SD}{S^\ddagger}
\newcommand{\SDp}{S^{\ddagger\prime}}
\newcommand{\MD}{M^\ddagger}
\newcommand{\MDp}{M^{\ddagger\prime}}

\newcommand{\sep}{,}

\usepackage{caption}
\begin{document}
\begin{abstract}
Let $B_t(n)$ be the number of set partitions of a set of size~$t$ into at most $n$ parts
and let $B'_t(n)$ be the number of
set partitions of $\{1,\ldots, t\}$ into at most $n$ parts
such that no part contains both $1$ and~$t$ or both $i$ and $i+1$ for any $i \in \{1,\ldots,t-1\}$.
We give two new combinatorial interpretations of the numbers $B_t(n)$ and $B'_t(n)$ using sequences
of random-to-top shuffles, %that leave a deck of cards invariant,
and sequences of box moves
on the Young diagrams of partitions. Using these ideas we obtain a very short proof of a
generalization of a result of Phatarfod on the eigenvalues of the random-to-top shuffle.
 We also prove analogous
results for random-to-top shuffles that may flip certain cards.
The proofs use the Solomon descent algebras of Types A, B and~D. We give
generating functions and asymptotic results for all the combinatorial quantities
studied in this paper.
\end{abstract}


\title[Bell numbers and partition moves]{Bell numbers, partition moves and the eigenvalues of the random-to-top shuffle in Dynkin Types A, B and D}
\date{\today}
\author{John R.~Britnell} 
\address{Department of Mathematics, Imperial College London, South Kensington Campus, London SW7 2AZ, United Kingdom}
\email{j.britnell@imperial.ac.uk}

\author{Mark Wildon}
\address{Department of Mathematics, Royal Holloway, University of London, Egham, Surrey TW20 0EX, United Kingdom}
\email{mark.wildon@rhul.ac.uk}
\maketitle
\thispagestyle{empty}

\section{Introduction}

For $t$, $n \in \N_0$, let $B_t(n)$ be the number of set partitions of $\{1,\ldots, t\}$
into at most $n$ parts. If $n \ge t$ then $B_t(n)$ is the Bell number $B_t$;
the difference $B_t(n)-B_t(n-1)$ is $\stir{t}{n}$, the Stirling number of the second kind.
Let $B'_t(n)$ be the number of set partitions of $\{1,\ldots, t\}$ into at most $n$ parts
such that no part contains both $1$ and~$t$ or both $i$ and $i+1$ for any $i \in \{1,\ldots,t-1\}$.
%The Stirling number of the second kind $\stir{t}{m}$ is the number of set partitions
%of $\{1,\ldots, t\}$ into exactly $m$ sets. Let $B_t(n) = \sum_{m=0}^n \stir{t}{m}$.
%Let $B'_t(n)$ be the number of st
%
%

The first object of this paper
is to give two combinatorial interpretations of the numbers $B_t(n)$
and $B'_t(n)$, one involving certain sequences of random-to-top shuffles, and another
involving
sequences of box removals and additions on the Young diagrams of partitions.
The first of these interpretation is justified by means of an explicit bijection. The second interpretation is considerably deeper, and its justification is less direct: our argument requires
the Branching Rule for representations of the symmetric group $\Sym_n$, and
a basic result from the theory of
Solomon's descent algebra.
Using these ideas we obtain a very short
proof of a generalization of a result due to Phatarfod~\cite{Phatarfod} on the eigenvalues of
the random-to-top shuffle.

We also state and prove analogous results for
random-to-top shuffles that may flip the moved card from face-up to face-down, using
the descent algebras associated to the Coxeter groups of Type~B and~D. In doing so, we introduce analogues of the Bell numbers corresponding to these types; these appear not to have been studied previously.
We give generating functions, asymptotic formulae and numerical relationships between these numbers,
and the associated Stirling numbers, in \S6 below.

We now define the quantities which we shall
show are equal to either $B_t(n)$ or $B'_t(n)$.

\begin{definition}
For $m \in \N$, let $\sigma_m$ denote the $m$-cycle $(1,2,\ldots, m)$.  A \emph{random-to-top
shuffle} of $\{1,\ldots, n\}$ is one of the $n$ permutations $\sigma_1, \ldots, \sigma_n$.
Let $S_t(n)$ be the number of sequences of
$t$ random-to-top shuffles whose product is the identity permutation.
Define $S'_t(n)$ analogously, excluding the identity permutation $\sigma_1$.
\end{definition}

We think of the permutations $\sigma_1, \ldots, \sigma_n$ as acting on the $n$ positions
in a deck of $n$ cards;
thus $\sigma_m$ is the permutation moving the card in position $m$ to
position $1$ at the top of the deck. If the cards are labelled by a set $C$
and the card in position $m$ is labelled by $c \in C$ then we
say that $\sigma_m$ \emph{lifts} card $c$.

We represent partitions by Young diagrams. Motivated by the Branching Rule,
we say that a box in a Young diagram is \emph{removable}
if removing it leaves the Young diagram of a partition; a position to which a box may be
added to give a Young diagram of a partition is said to be \emph{addable}.

\begin{definition}
A \emph{move} on a partition consists of the removal of a removable box and
then addition in an addable position of a single box. A move is \emph{exceptional} if it consists
of the removal and then addition in the same place of the lowest removable box.
Given partitions $\lambda$ and $\mu$ of the same size
let $M_t(\lambda, \mu)$ be the number of sequences
of~$t$ moves that start at $\lambda$ and finish at $\mu$.
%Changed so that moves are on partitions, not Young diagrams
Let $M_t'(\lambda,\mu)$ be defined analogously, considering
only non-exceptional moves. For $n \in \N_0$, let $M_t(n) = M_t\bigl( (n), (n) \bigr)$
and let $M'_t(n) = M'_t\bigl( (n), (n) \bigr)$.
\end{definition}

We note that if the Young diagram of $\lambda$ has exactly $r$ removable boxes
then $M_1(\lambda,\lambda) = r$ and $M'_1(\lambda,\lambda) = r-1$. %since ...
For example, $(2,1)$ has moves to~$(2,1)$, in two ways, and also to $(3)$, $(1^3)$.

Our first result is as follows.

\begin{theorem}\label{thm:main}
For all $t$, $n \in \N_0$ we have $B_t(n) = S_t(n) = M_t(n)$ and $B'_t(n) = S'_t(n) = M'_t(n)$.
\end{theorem}

In \S 3 we consider the more general $k$-shuffles for $k \in \N$ and use
our methods to give a short proof of Theorem~4.1 in \cite{DFP}
on the eigenvalues of the associated Markov chain.
Theorem~\ref{thm:main} and this result have analogues for Types B and D. We state and prove these
results in \S\ref{sec:typeB} and \S\ref{sec:typeD}.

%[{\bf Further results in which pack is reversed after shuffle, corresponding to transposing partitions?}]
We remark that the sequences of partition moves counted by $M'_t\bigl( \lambda, \mu \bigr)$
are in  bijection with the Kronecker tableaux $KT^t_{\lambda, \mu}$
defined in \cite[Definition~2]{GoupilChauve}. In \S 7
we show that, when $n \ge t$,
the equality $B'_t(n) = M'_t(n)$ is a corollary of a special case
of~\cite[Lemma~1]{GoupilChauve}. The proof of this lemma depends on results on
the RSK correspondence for oscillating tableaux, as developed in~\cite{DDF} and \cite{Sundaram}.
Our proof, %of the equality,
which goes via the numbers $S'_t(n)$, is
different and  significantly shorter.
%In the final section
%We then prove Phatarfod's result.
In \S7 we also
 make some remarks on the connections with earlier work
 of Bernhart \cite{Bernhart} and
%Diaconis, Fill and Pitman \cite{DFP} and
Fulman \cite{FulmanCardShuffling} and
show an unexpected obstacle to a purely bijective proof of Theorem~\ref{thm:main}:
this line of argument is motivated by \cite{FulmanCardShuffling} and \cite{GoupilChauve}.
We end in \S8 by discussing appearances of the numbers studied in this paper in OEIS \cite{OEIS};
with the expected exception of $B_t(n)$ and $\stir{t}{n}$ these are few, and occur for
particular choices of the parameters $t$ and $n$.

%The Appendix to the paper gives some tables and references to OEIS \cite{OEIS}.

%{\bf Submit sequences to OEIS? See notes in TeX file.}

%Mention of OEIS?
%$\stir{t}{2}' = B'_t(2)$ is A000296: 0, 1, 0, 1, 0, 1, 0, 1, 0, 1, ..
%$\stir{t}{3}'$ is A000975: 0, 0, 1, 2, 5, 10, 21, 42, 85, 170, ...
%$\stir{t}[4}'$ is A243869: 0, 0, 0, 1, 5, 20, 70, 231, 735, 2290
% connection with a Knuth AMM problem about people on a circular table, 4 groups

%$\stir{t}{5}' and \stir{t}{6}' aren't in there

%B_2(n) = 2^{n-1}$ and
%B_3(n) = (3^n+1)/2$ for all $n \in \N$.

%$B'_t(t)$ is A000296

%$B'_t(3)$ is A152046
%'A product based sequence: a(n)=Product[(1 + 8*Cos[k*Pi/n]^2), {k, 1, Floor[(n - 1)/2]}].'
%$B'_t(4)$ is not in there

%Type B and D quantities: will add to the MAGMA code and compute some more numbers
%but I don't expect to them find in OEIS

\section{Proof of Theorem~\ref{thm:main}}

We prove the first equality in Theorem~\ref{thm:main} using
an explicit bijection.

\begin{lemma}\label{lemma:BS}
If $t$, $n \in \N_0$ then $B_t(n) = S_t(n)$ and $B'_t(n) = S'_t(n)$.
\end{lemma}

\begin{proof}
Suppose that $\tau_1, \ldots, \tau_t$ is a sequence of top-to-random shuffles such
that $\tau_1 \ldots \tau_t = \id_{\Sym_n}$.  Take a deck of
cards labelled by $\{1,\ldots, n\}$, so that card $c$ starts in position $c$,
and apply the shuffles so that at time $s \in \{1,\ldots, t\}$ we
permute the positions of the deck by $\tau_s$.
For each $c \in \{1,\ldots, n\}$ let
$A_c$ be the set of $s \in \{1,\ldots, t\}$
such that $\tau_s$ lifts  card~$c$.
%Let $f(\tau_1, \ldots, \tau_t)$
%be the set partition of $\{1,\ldots, t\}$ obtained by removing any empty
%sets from $A_1, \ldots, A_n$.
Removing any empty sets from the list $A_1, \ldots, A_n$ we obtain
a set partition of $\{1,\ldots, t\}$ into at most $n$ sets.

Conversely, given a set partition of $\{1,\ldots, t\}$ into $m$ parts where \hbox{$m \le n$},
we claim that there is a unique way to label its parts $A_1, \ldots, A_m$,
and a unique sequence of $t$ random-to-top shuffles leaving the deck invariant, such that~$A_c$ is
the set of times when card $c$ is lifted by the shuffles in this sequence.
If such a labelling exists, then for each $c \in \{1,\ldots, m\}$,
the set~$A_c$ must contain the greatest element of $\{1,\ldots, t\} \backslash
(A_1\cup \cdots \cup A_{c-1})$, since otherwise a card $c'$ with $c' > c$ is lifted
after the final time when card $c$ is lifted, and from this time onwards, cards~$c$
and~$c'$ are in the wrong order.
Using this condition to fix $A_1, \ldots, A_m$ determines
the card lifted at each time, and so determines a unique sequence of random-to-top shuffles
that clearly leaves the deck invariant. It follows that $B_t(n) = S_t(n)$.

A set partition corresponds to a sequence of non-identity shuffles if and only
if the same card is never lifted at consecutive times, and the first card lifted
is not card $1$. Therefore the bijection just defined restricts to give
$B'_t(n) = S'_t(n)$.
\end{proof}

The second half of the proof is algebraic.
Let $\Cl(\Sym_n)$ denote the ring of class
functions of $\Sym_n$. Let $\pi \in \Cl(\Sym_n)$ be the natural permutation
character of $\Sym_n$, defined by $\pi(\tau) = |\Fix\ \tau|$ for $\tau \in \Sym_n$.
Let $\chi^\lambda$ denote the irreducible character of $\Sym_n$ canonically
labelled by the partition $\lambda$ of~$n$. Let $\theta = \pi - 1_{\Sym_n}$
where $1_{\Sym_n}$ is the trivial character of $\Sym_n$; note that $\theta = \chi^{(n-1,1)}$.
The main idea in the following lemma is well known: it appears in
\cite[Lemma 4.1]{BessenrodtKleshchev}, and the special case
for $M'_t\bigl( (n),\mu \bigr)$ is proved in
\cite[Proposition 1]{GoupilChauve}.

\begin{lemma}\label{lemma:branching}
Let $t \in \N_0$. If $\lambda$ and $\mu$ are partitions of $n \in \N$ then
$M_t(\lambda, \mu) = \langle \chi^\lambda \pi^t, \chi^\mu \rangle$ and
$M'_t(\lambda,\mu) = \langle \chi^\lambda \theta^t, \chi^\mu \rangle$.
\end{lemma}

\begin{proof}
We have
\[ \chi^\lambda \pi = \chi^\lambda \bigl(1\Ind_{\Sym_{n-1}}^{\Sym_n} \bigr) =
\bigl( \chi^\lambda \Res_{\Sym_{n-1}}\bigr) \Ind^{\Sym_n}. \]
By the Branching Rule for $\Sym_n$ (see \cite[Chapter~9]{James}) we have
$\chi^\lambda \hskip-0.5pt\res_{\Sym_{n-1}} = \sum_\nu \chi^\nu$ where the sum
is over all partitions $\nu$ whose Young diagram is obtained from $\lambda$ by
removing a single box. The case $t=1$ of the first part of the
lemma now follows by Frobenius reciprocity, and the general case follows by induction.
The second part is proved similarly, using that
$\chi^\lambda \theta = \chi^\lambda (\pi - 1_{\Sym_n}) =
\bigl( \chi^\lambda \res_{\Sym_{n-1}}\bigr) \ind^{\Sym_n}\! -{}\hskip1pt \chi^\lambda$.
\end{proof}

Let $\D(\Sym_n)$ be the descent subalgebra of the rational group algebra $\Q\Sym_n$,
as defined in \cite[page 7]{BlessenohlSchocker} or \cite[Theorem~1]{Solomon}. %(replacing $\Z$ with $\Q$).
Let $\Xi = \sum_{m=1}^n \sigma_m^{-1}$ and let $\Delta = \Xi - \id_{\Sym_n}$. %\sum_{m=2}^n \sigma_m$.
Note that $\tau \in \Sym_n$ is a summand of $\Delta$ if and only if
\[ 1 \tau > 2 \tau < 3 \tau < \ldots < n\tau. \]
Thus $\Delta$ is the sum of all $\tau \in \Sym_n$ such that the unique descent of $\tau$ is in position $1$.
Hence $\Delta$, $\Xi \in \D(\Sym_n)$. By \cite[Theorem~1.2]{BlessenohlSchocker} or \cite[Theorem~1]{Solomon},
under the canonical
algebra epimorphism $\D(\Sym_n) \rightarrow \Cl(\Sym_n)$, we have $\Xi \mapsto \pi$, $\id_{\Sym_n} \mapsto
1_{\Sym_n}$ and $\Delta \mapsto \theta$. Define a bilinear form on $\Q\Sym_n$ by
$(g,h) = 1$ if $g = h^{-1}$ and $(g,h) = 0$ if $g \not= h^{-1}$.
By \cite[Theorem~1.2]{BlessenohlSchocker},
this epimorphism is an isometry with respect
to the bilinear form on $\D(\Sym_n)$ defined by restriction of $(-,-)$.

%Add isometry, or something to justify extracting
%coefficient: what we need is
%[id] (elt of descent algebra) = \langle 1_{\Sym_n}, image in Cl(S_n)\rangle, which holds
%because the coefficient of the identity is the inner product in Des with \Xi^{(n)}.

\begin{lemma}\label{lemma:SM}
If $t$, $n \in \N_0$  then $S_t(n) = M_t(n)$ and $S'_t(n) = M'_t(n)$.
\end{lemma}

\begin{proof}
If $n = 0$ then clearly $S_t(0)=M_t(0)=S'_t(0)=M'_t(0)$; the common value is $0$ if $t \in \N$
and $1$ if $t=0$. Suppose that $n \in \N$.
By Lemma~\ref{lemma:branching} it is sufficient to prove that
$\langle \pi^t, 1_{\Sym_n} \rangle = S_t(n)$ and $\langle \theta^t, 1_{\Sym_n}\rangle = S'_t(n) $.
Write $[\id_{\Sym_n}]$ for the coefficient of the identity permutation in an element
of $\Q \Sym_n$. By the remark before this lemma, if $\Gamma \in \D(\Sym_n)$ maps to $\phi \in \Cl(\Sym_n)$
under the epimorphism  $\D(\Sym_n) \rightarrow \Cl(\Sym_n)$, then
\begin{equation}\label{eq:iso}
[\id_{\Sym_n}]\Gamma = (\Gamma, \id_{\Sym_n}) = \langle \phi, 1_{\Sym_n} \rangle.
\end{equation}
The two required results now follow from Lemma~\ref{lemma:branching} using the obvious equalities
$[\id_{\Sym_n}] \Xi^t = S_t(n)$ and $[\id_{\Sym_n}] \Delta^t = S'_t(n)$.
\end{proof}

Theorem~\ref{thm:main} now follows from Lemmas~\ref{lemma:BS} and~\ref{lemma:SM}.

\section{Eigenvalues of the $k$-shuffle}\label{sec:Phatarfod}

Fix $n$, $k \in \N$ with $k \le n$. A $k$-shuffle of a deck of $n$ cards
takes any $k$ cards in the deck and moves them to the top of the deck, preserving their
relative order. Note that a $1$-shuffle is a random-to-top shuffle as already defined
and that the inverse of a $k$-shuffle is a riffle shuffle involving the top~$k$ cards.
%Thus a $1$-fold
%random-to-top shuffle is a random-to-top shuffle as already defined.
 Let $P(k)$ be the %$n! \times n!$-
transition matrix of the Markov chain on $\Sym_n$ in which each step is
 given by multiplication by one of
 the $\binom{n}{k}$ $k$-shuffles, chosen
uniformly at random.
Thus for $\sigma$, $\tau \in \Sym_n$ we have
\[ P(k)_{\sigma\tau} = \begin{cases} \binom{n}{k}^{-1} & \text{if $\sigma^{-1}\tau$ is
a $k$-shuffle}
% \sigma_m = \tau^\star$ for some $m \in \{1,\ldots, n\}$}
\\ 0 & \text{otherwise.} \end{cases}. \]

It was proved by Phatarfod in \cite{Phatarfod} that
the eigenvalues of $P(1)$ are exactly the numbers $|\Fix\ \tau|/n$ for $\tau \in \Sym_n$.
More generally, it follows from a statement
by Diaconis, Fill and Pitman \cite[(6.1)]{DFP}
that
if~$\pi_k(\tau)$ is the
number of fixed points of $\tau \in \Sym_n$ in its action on $k$-subsets of $\{1,\ldots, n\}$,
then the eigenvalues of~$P(k)$
are exactly the numbers $\pi_k(\tau)/\binom{n}{k}$ for $\tau \in \Sym_n$.
The proof of this statement in \cite{DFP} is referred to unpublished work of Diaconis, Hanlon and Rockmore.
We provide a short proof here.

%Below we give a very short proof of the later generalization in \cite[Theorem 4.1]{DFP}.
Observe that $(\Tr P(k)^t)/n!$ is the probability that $t$ sequential
$k$-shuffles leave the deck invariant.
It is easily seen that $\tau^{-1}$ is a $k$-shuffle if and only if
$1 \tau < \ldots < k \tau$ and $(k+1)\tau < \ldots < n\tau$.
Hence the sum of the inverses of the $k$-shuffles is the
basis element $\Xi^{(k,n-k)} \in \D(\Sym_n)$ defined in \cite[page 7]{BlessenohlSchocker},
which maps to $\pi_k = 1_{\Sym_k \times \Sym_{n-k}}\!\ind^{\Sym_n}$ under the canonical
epimorphism $\D(\Sym_n) \rightarrow \Cl(\Sym_n)$. % by  \cite[Theorem 1.2]{BlessenohlSchocker}.
Therefore~\eqref{eq:iso} in the proof of Lemma~\ref{lemma:SM} implies that
\begin{align*}
 \frac{\Tr P(k)^t}{n!} &= [\mathrm{id}_{\Sym_n}]  \binom{n}{k}^{-t}
 %\bigl(
{\Xi^{(k,n-k)}}^t % \bigr)^t
 \\
&= \binom{n}{k}^{-t} \langle \pi_k^t, 1_{\Sym_n} \rangle \\
&= \binom{n}{k}^{-t} \frac{1}{n!} \sum_{\tau \in \Sym_n}
%\Bigl(
\pi_k(\tau)^t
%\frac{\pi_k(\tau)}{\binom{n}{k}} \Bigr)^t
\end{align*}
for all $t \in \N_0$. It follows that if $\epsilon_1, \ldots, \epsilon_{n!}$ are the eigenvalues
of $P(k)$, then $\sum_{i=1}^{n!} \epsilon_i^t = \sum_{\tau \in \Sym_n}
\pi_k(\tau)^t/\binom{n}{k}^t $
for all $t \in \N_0$. Thus the multisets $\{\epsilon_i : 1 \le i \le n!\}$ and
$\{ \pi_k(\tau) / \binom{n}{k} : \tau \in \Sym_n \}$ are equal, as required.

The analogous result for the $k$-shuffle with the identity permutation
excluded is stated below.

\begin{proposition}\label{prop:eigenvaluesA}
Let $n \ge 2$ and
let $P'(k)$ be the transition matrix of the $k$-shuffle, modified
so that at each time one of the $\binom{n}{k}-1$ non-identity
$k$-shuffles is chosen uniformly at random.
The eigenvalues
of $P'(k)$ are $(\pi_k(\tau) - 1) / (\binom{n}{k}-1)$
for $\tau \in \Sym_n$. %\hfill$\qed$
\end{proposition}

\begin{proof}
The proposition may be proved by adapting our proof of the result of Diaconis, Fill and Pitman.
Alternatively,
it may be obtained as a straightforward corollary to that result, by observing that \[\Bigl( \binom{n}{k}-1 \Bigr)P'(k)
= \binom{n}{k}P(k) - I_n.\qedhere\]
\end{proof}

We note that the eigenvalues of the
$k$-top-to-random shuffles considered in Theorem~4.1 and the following Remark 1 in \cite{DFP}
may be determined by a  similar short argument using the element $\Xi^{(1,\ldots,1,n-k)} \in \D(\Sym_n)$.

%{\bf Second largest eigenvalue?}

\section{Oriented random-to-top shuffles: Type B}\label{sec:typeB}

In this section we state and prove
the analogue of Theorem~\ref{thm:main} for the Coxeter group %$C_2 \wr S_n$
of Type~B.
%In type D
%there is no obvious analogue of the numbers $B_t(n)$ and $B'_t(n)$, but Lemma~\ref{lemma:SM} still
%has an analogue, which we give below.
Henceforth we shall always use the symbol $\dagger$ %and $\ddagger$
%are used
to indicate quantities relevant to this type. % or type D, respectively.

\subsection{Type B set partitions, shuffles and partition moves}\label{sec:typeBstatement}

We begin by defining the three quantities that we shall show are equal. Let $t$, $n \in \N_0$.

%For $t$, $n \in \N_0$,
Let $\BB_t(n)$ be the number of set partitions of $\{1,\ldots, t\}$ into at most $n$ parts with
an even number of elements in each part distinguished by \emph{marks}. Let $\BBp_t(n)$ be defined similarly,
counting only those set partitions such that if  $1$ and $t$ are in the same part then $1$ is marked,
and if both $i$ and $i+1$ are in the
same part then $i+1$ is marked.
%(Thus if $t = 1$ then $\BBp_1(0) = 1$ and $\BBp_1(n) = 0$
%for all $n \in \N$, since~$1$ must be marked if it is present.)

We define the Coxeter group $\BSym_n$  to be the maximal subgroup
of the symmetric group on the set $\{-1,\ldots, -n\} \cup \{1,\ldots, n\}$
that permutes $\{-1,1\}, \ldots, \{-n,n\}$ as blocks for its action. Let $\sigma \in \BSym_n$.
We interpret~$\sigma$  as the shuffle of a deck of $n$ cards, each oriented either \emph{face-up} or
\emph{face-down}, that moves a card in position $i \in \{1,\ldots, n\}$
to position $|i \sigma|$, changing the orientation of the card
if and only if $i\sigma \in \{-1,\dots, -n\}$.

\begin{definition}
For $m \in \N$ let
\begin{align*}
\rho_m &= (-1,\ldots, -m)(1,\ldots, m), \\
\bar{\rho}_m &= (-1, \ldots, -m, 1, \ldots, m).
\end{align*}
We say that $\rho_m$ and $\bar{\rho}_m$ are \emph{oriented
random-to-top shuffles}.
Let $\SB_t(n)$ be the number of sequences of $t$ oriented random-to-top shuffles whose product
is the identity permutation.
Define $\SBp_t(n)$ analogously, excluding the identity permutation~$\rho_1$.
\end{definition}

Observe that, under our card shuffling interpretation,
$\bar{\rho}_m$ lifts the card in position $m$ to the top of the deck and then flips it.

A pair $(\lambda, \lambda^\star)$ of partitions such that the
sum of the sizes of $\lambda$ and $\lambda^\star$ is~$n \in \N_0$ will be
called a \emph{double partition of $n$}.
By \cite[Theorem 4.3.34]{JK}
the irreducible characters of $\BSym_n$ are
canonically labelled by double partitions of~$n$. We briefly
recall this construction. Given $m \in \N$,
let $\widetilde{\sgn}^{\times m}$ denote the linear character of $C_2 \wr \Sym_m$ on which
each factor of the base group $C_2 \times \cdots \times C_2$ acts as $-1$.
Let $\Inf$ denote the inflation functor. Define
\[ \chi^{(\lambda \sep \lambda^\star)} =
\bigl( \Inf^{C_2 \wr \Sym_\ell}_{\Sym_\ell} \chi^\lambda \,\times\, \widetilde{\sgn}^{\times \ell^\star}
\Inf^{C_2 \wr \Sym_{\ell^\star}}_{\Sym_{\ell^\star}} \chi^\mu \bigr) \Ind^{\BSym_n} \]
where $\lambda$ is a partition of $\ell$ and $\lambda^\star$ is a partition of $\ell^\star$.
The $\chi^{(\lambda,\lambda^\star)}$ for $(\lambda,\lambda^\star)$ a double partition of $n$
are exactly the irreducible characters of $\BSym_n$.

\begin{definition}
Let $(\lambda, \lambda^\star)$ be a double partition. A \emph{double-move}
on $(\lambda, \lambda^\star)$ consists of the removal and then
addition of a single box on the corresponding Young diagrams.
(The box need not be added to the diagram from which it is removed.)
A double-move is \emph{exceptional} if it consists of the removal and then
addition in the same place of the lowest removable box in the diagram of $\lambda$,
or if $\lambda$ is empty, of the lowest removable box in the diagram of $\lambda^\star$.
Let $\MB_t\bigl( (\lambda, \lambda^\star), (\mu,\mu^\star) \bigr)$
be the number of sequences
of~$t$ double-moves that start at $(\lambda, \lambda^\star)$ and finish at $(\mu,\mu^\star)$.
Let $\MBp_t\bigl( (\lambda, \lambda^\star), (\mu,\mu^\star) \bigr)$ be defined analogously, considering
only non-exceptional double-moves. Let $\MB_t(n) = \MB_t\bigl( ((n), \varnothing),
((n), \varnothing) \bigr)$
and let $\MBp_t(n) = \MBp_t\bigl( ((n), \varnothing), ((n), \varnothing) \bigr)$.
\end{definition}

%Omit?
This definition of `exceptional' is convenient for small examples, but can be replaced
with any other that always excludes exactly one double-move from the
set counted by $\MB_1\bigl( (\lambda, \lambda^\star), (\lambda,\lambda^\star) \bigr)$.

We can now state the analogue of Theorem~\ref{thm:main} for Type B.

\begin{theorem}\label{thm:mainB}
For all $t$, $n \in \N_0$ we have
$B^\dagger_t(n) = S^\dagger_t(n) = M^\dagger_t(n)$ and
$B^{\dagger\prime}_t(n) = S^{\dagger\prime}_t(n) = M^{\dagger\prime}_t(n)$.
%$\BB_t(n) = \SB_t(n) = \MB_t(n)$ and $\BB'_t(n) = \SB'_t(n) = \MB'_t(n)$.
\end{theorem}

The proof of this theorem follows the same general plan as the proof of Theorem~\ref{thm:main}
so we shall present it quite briefly. The most important difference is that we must make
explicit use of the underlying root system.

\subsection{Proof of Theorem~\ref{thm:mainB}}
The proof of Lemma~\ref{lemma:BS} used a bijection between
set partitions of $\{1,\ldots, t\}$ into at most $n$ parts and sequences
of~$t$ random-to-top shuffles leaving a deck of $n$ cards invariant. This bijection
can be modified to give
a bijection between the marked set partitions counted by $\BB_t(n)$ and the shuffle
sequences counted by $\SB_t(n)$: given a marked set partition~$\mathcal{P}$,
the corresponding shuffle flips the orientation of the card lifted at time $s \in \{1,\ldots, t\}$
if and only if $s$ is a marked element of~$\mathcal{P}$. This map restricts to
 a bijection between the marked set partitions counted by $\BBp_t(n)$ and the shuffle
sequences counted by $\SBp_t(n)$. Thus $\BB_t(n) = \SB_t(n)$ and $\BBp_t(n) = \SBp_t(n)$
for all $t$, $n \in \N_0$.

%Let $\chi^{(\lambda,\lambda^\star)}$ denote the irreducible character of
%$\BSym_n$ canonically labelled by the double partition $(\lambda,\lambda^\star)$ of $n$.
Let $\BSym_{n-1} = \{ \sigma \in \BSym_{n} : 1 \sigma = 1\}$ %, (-1)\sigma = -1 \}$
and let $\pi^\dagger = 1_{\BSym_{n-1}} \ind^{\BSym_n}$. Inducing via the
subgroup $\langle (-1,1) \rangle \times \BSym_{n-1}$ of $\BSym_n$ one finds that
\[ \pi^\dagger = \chi^{( (n) \sep \varnothing )} + \chi^{( (n-1,1) \sep \varnothing)}
 + \chi^{( (n-1) \sep (1) )}.\]
Let $\theta^\dagger = \pi^\dagger - 1_{\BSym_n}$.
The analogue of Lemma~\ref{lemma:branching} is as follows.

\begin{lemma}\label{lemma:branchingB}
If $(\lambda, \lambda^\star)$ and $(\mu, \mu^\star)$ are double partitions of $n \in \N$ then
\begin{align*}
\MB_t\bigl( (\lambda, \lambda^\star), (\mu,\mu^\star) \bigr) &= \langle \chi^{(\lambda \sep \lambda^\star)}
\pi^t, \chi^{(\mu \sep  \mu^\star)} \rangle, \\
\MBp_t\bigl( (\lambda, \lambda^\star), (\mu,\mu^\star) \bigr) &=
\langle \chi^{(\lambda \sep \lambda^\star)} \theta^t, \chi^{(\mu \sep \mu^\star)} \rangle.
\end{align*}
%$\MB_t\bigl( (\lambda, \lambda^\star), (\mu,\mu^\star) \bigr) = \langle \chi^{(\lambda, \lambda^\star)}
%\pi^t, \chi^{(\mu, \mu^\star)} \rangle$ and
%$\MB'_t\bigl( (\lambda, \lambda^\star), (\mu,\mu^\star) \bigr) =
%\langle \chi^{(\lambda, \lambda^\star)} \theta^t, \chi^{(\mu, \mu^\star)} \rangle$.
\end{lemma}

\begin{proof}
The Branching Rule for $\BSym_n$ is
$\chi^{(\lambda \sep \lambda^\star)}\!\!\res_{\BSym_{n-1}}
= \sum_{(\nu,\nu^\star)} \chi^{(\nu \sep \nu^\star)}$ where the sum is over all double partitions
$(\nu,\nu^\star)$ obtained by removing a single box from either the Young diagram of $\lambda$
or the Young diagram of $\lambda^\star$.
(It is routine to prove this using the theory in \cite[\S 4.3]{JK} and
the Branching Rule for $\Sym_n$.)
The case $t=1$ of the first part of the lemma now
follows by Frobenius reciprocity, and the second part is proved similarly.
The general case then follows by induction on $t$.
\end{proof}

The group $\BSym_n$ acts on the $n$-dimensional real vector
space $E = \langle e_1, \ldots, e_n \rangle$ by $e_i \sigma = \pm e_j$ where $j = |i \sigma|$
and the sign is the sign of $i\sigma$.
Let $\Phi^+ = \{e_i - e_j : 1 \le i < j \le n \} \cup \{e_i + e_j : 1 \le i < j \le n\}
\cup \{ e_k : 1 \le k \le n\}$. Then $\Phi^+$
is the set of positive roots in a root system of type B, having simple roots
$\alpha_i = e_i - e_{i+1}$ for $1 \le i < n$
and $\beta = e_n$.
Let $\D(\BSym_n)$ be the descent subalgebra of the rational group algebra $\Q\!\BSym_n$,
as defined in \cite[Theorem~1]{Solomon}, for this choice of simple roots. Thus $\D(\BSym_n)$ has
as a basis the elements $\Xi^\dagger_I$ defined by
\[ \Xi_I^\dagger = \bigl\{ \sigma \in \BSym_n : \{ \alpha \in \Phi^+ : \alpha \sigma \in -\Phi^+ \}
\subseteq I \} \bigr\} \]
for $I \subseteq \{ \alpha_1, \ldots, \alpha_{n-1}, \beta \}$.

 % (replacing $\Z$ with $\Q$).
We need the following basic lemma describing descents in Type B for our choice of simple roots.

\begin{lemma}\label{lemma:interval}
Let $\sigma \in \BSym_n$. If $\ell$ and $m$ are such that $\alpha_\ell \sigma, \ldots,
\alpha_{m-1} \sigma \in \Phi^+$ then either $\ell \sigma, \ldots, m\sigma$ all have the same sign
and $\ell \sigma < \ldots < m \sigma$,
 or there exists a unique
$q \in \{\ell, \ldots, m-1\}$ such that $1 \le \ell \sigma < \ldots < q \sigma$
and $(q+1) \sigma < \ldots < m \sigma \le -1$.
\end{lemma}

\begin{proof}
It is routine to check that $\alpha_i \sigma \in \Phi^+$ if and only if one of
\begin{itemize}
\item[(a)] $i \sigma \in \{1,\ldots, n\}$, $(i+1) \sigma \in \{1,\ldots, n\}$ and $i \sigma < (i+1) \sigma$;
\item[(b)] $i \sigma \in \{1,\ldots, n\}$, $(i+1) \sigma \in \{-1,\ldots, -n\}$;
\item[(c)] $i \sigma \in \{-1,\ldots,-n\}$, $(i+1) \sigma \in \{-1,\ldots, -n\}$ and
$i\sigma < (i+1)\sigma$.
\end{itemize}
The lemma now follows easily.
\end{proof}

Let $\Xi^\dagger = \sum_{m=1}^n (\rho_m^{-1} + {\bar{\rho}^{-1}_m})$ and let $\Delta^\dagger = \Xi - \id_{\BSym_n}$.
%We now show that $\Xi^\dagger$ and $\Delta^\dagger$ are elements
%of $\D(\BSym_n)$.

\begin{lemma}\label{lemma:XiB}
Let $n \ge 2$. Then
$\Xi^\dagger$, $\Delta^\dagger \in \D(\BSym_n)$ and under the canonical algebra epimorphism
$\D(\BSym_n) \rightarrow \Cl(\BSym_n)$ we have $\Xi^\dagger \mapsto \pi^\dagger$ and
$\Delta^\dagger \mapsto \pi^\dagger - 1_{\Sym_n}$.
\end{lemma}

\begin{proof}
By Lemma~\ref{lemma:interval}, if
%The reflections in the simple roots act on $\Omega_n$ by
%$s_{\alpha_i} = (i, i+1)(-i,-(i+1))$ for $1 \le i \le n-1$ and $s_\beta = (-n, n)$.
 $\alpha_i \sigma \in \Phi^+$ for $2 \le i < n$
then there exists a unique $p \in \{1,\ldots, n\}$ such that $\{2,\ldots, p\} \sigma \subseteq \{1,\ldots, n\}$ and
$\{p+1,\ldots, n\} \sigma \subseteq \{-1,\ldots, -n\}$. If in addition $\beta \sigma \in \Phi^+$,
we must have $p = n$. Therefore
$\alpha_1$ is the unique simple root $\delta$ such
that $\delta \sigma \not\in \Phi^+$ if and only if $\{2,\ldots, n\}\sigma \subseteq \{1,\ldots, n\}$
and \emph{either}
\begin{itemize}
\item[(1)] $1 \sigma \in \{1,\ldots, n\}$ and $1\sigma > 2 \sigma < \ldots < n \sigma$, \emph{or}
\item[(2)] $1 \sigma \in \{-1,\ldots, -n\}$ and $2 \sigma < \ldots < n\sigma$.
\end{itemize}
The permutations in Case (1) are $\{\rho_2^{-1}, \ldots, \rho_n^{-1} \}$ and the permutations
in Case (2) are $\{\bar{\rho}_1^{-1}, \ldots, \bar{\rho}_n^{-1} \}$. Hence $\Xi^\dagger
= \Xi^\dagger_{\{\alpha_1\}}$. This proves the first part of the lemma.

%The subgroup is generated by
The reflections in the simple roots $\alpha_2, \ldots, \alpha_{n-1}$ and $\beta$ are
$s_{\alpha_i} = (i,i+1)(-i,-(i+1))$ for $2 \le i \le n-1$ and $s_\beta = (-n,n)$.
These generate the subgroup  $\BSym_{n-1}$ defined above. Hence $\BSym_{n-1}$ is a parabolic
subgroup of $\BSym_n$ and
the permutation character of $\BSym_{n}$ acting on the cosets of $\BSym_{n-1}$
is $\pi^\dagger$.
%already seen.
By \cite[Theorem~1]{Solomon}, under the canonical algebra epimorphism
$\D(\BSym_n) \rightarrow \Cl(\BSym_n)$, we have $\Xi^\dagger \mapsto \pi^\dagger$. This completes the proof.
\end{proof}

We are now ready to prove the second equality in Theorem~\ref{thm:mainB}. %$^\dagger$.
We require the fact that the
epimorphism $\D(\BSym_n) \rightarrow \Cl(\BSym_n)$ is an isometry
with respect to the bilinear form on $\D(\BSym_n)$ defined by restriction of the
form $(-,-)$ on $\Q\Sym_{2n}$; %defined before Lemma~\ref{lemma:SM};
this is proved in
\cite[Theorem 3.1]{HohlwegSolomon}.

\begin{lemma}\label{lemma:daggerSM} %{lemma:SM}
For $t$, $n \in \N_0$ we have $\SB_t(n) = \MB_t(n)$ and $\SBp_t(n) = \MBp_t(n)$.
\end{lemma}

\begin{proof}
If $n=0$ the result is clear.
When $n = 1$ the oriented top-to-random shuffles are the identity and $(-1,1)$
and it is easily checked that $\SB_0(1) = \MB_0(1) = 1$ and
$\SB_t(1) = \MB_t(1) = 2^{t-1}$ for all $t \in \N$. Similarly
$\SBp_t(1) = \MBp_t(1) = 1$ if $t$ is even and $\SBp_t(1) = \MBp_t(1) = 0$ if $t$ is odd.
When $n \ge 2$ we follow the
proof of Lemma~\ref{lemma:SM}, using that
$[\id_{\BSym_n}] \Xi^{\dagger t} = \SB_t(n)$ and $[\id_{\BSym_n}] \Delta^{\dagger t} = \SBp_t(n)$.
\end{proof}

This completes the proof of Theorem~\ref{thm:mainB}. %$^\dagger$.

%
%For example, in the primed cases, the sets counted by ${S'_3}^B(3) = {M'_3}^B(3) = {B'_3}^B(3) = 4$ are
%\[ \left\{ \begin{matrix}
%\bigl((1\bar{1}), (1 \bar{2}\bar{1}2), (12)(\bar{1}\bar{2})\bigr), \\
%\bigl((1 \bar{2}\bar{1}2), (12)(\bar{1}\bar{2}), (1 \bar{1})\bigr), \\
%\bigl((12)(\bar{1}\bar{2}), (1 \bar{1}), (1 \bar{2}\bar{1}2)\bigr), \\
%\bigl((132)(\bar{1}\bar{3}\bar{2}), (132)(\bar{1}\bar{3}\bar{2}), (132)(\bar{1}\bar{3}\bar{2})\bigr)
%\end{matrix} \right\} \]
%and
%\[ \left\{ \begin{matrix} (3) \rightarrow (2,1) \rightarrow (2,1) \rightarrow (3), \\
%(3) \rightarrow (2), (1) \rightarrow (2), (1) \rightarrow (3),\\
%(3) \rightarrow (2), (1) \rightarrow (2,1) \rightarrow (3)\\
%(3) \rightarrow (2,1) \rightarrow (2), (1) \rightarrow (3) \end{matrix} \right\}
%\]
%where empty partitions are suppressed and the (unique) non-exceptional move is chosen when necessary,
%and
%\[ \Bigl\{ \bigl\{ \{1\}, \{2\}, \{3\} \bigr\},
%\bigl\{ \{1\}, \{\bar{2}, \bar{3} \} \bigr\},
%\bigl\{ \{\bar{1}, \bar{2} \}, \{ 3 \} \bigr\},
%\bigl\{ \{\bar{1}, \bar{3} \}, \{2\} \bigr\} \]
%corresponding to the sequences of lifted cards
%\[ \{ 321, 2\bar{1}\bar{1}, \bar{2}\bar{2}1, \bar{1}2\bar{1} \} \]
%where bars indicate marks, i.e.~the lifted card is flipped.
%(In fact it is probably better to work with the inverse permutations, thought of as position permutations;
%the permutations used above are the summands of descent algebra element for the simple root $\alpha_1$.)
%

\subsection{Eigenvalues of the oriented $k$-shuffle}
%Fix $n$, $k \in \N$ with $k \le n$.
%As in \S\ref{sec:Phatarfod}, we now consider a more general family of permutations.
The analogue for Type~B of the~$k$-shuffle is most conveniently
defined using our interpretation of elements
of $\BSym_n$ as shuffles (with flips) of a deck of $n$ cards.

\begin{definition}
Let $n$, $k \in \N$ with $k \le n$.
An \emph{oriented $k$-shuffle} is performed as follows.
Remove $k$ cards from the deck. Then choose any $j \in \{0,1,\ldots,k\}$ of the
$k$ cards
and flip these $j$ cards  over as a block.
Place the $j$ flipped cards on top of the deck, and then put the $k-j$
unflipped cards on top of them.
\end{definition}

Thus there are $2^k\binom{n}{k}$ oriented $k$-shuffles.
After an oriented $k$-shuffle that flips $j$ cards, the newly
flipped cards occupy positions $k-j+1, \ldots, k$, and appear in the reverse of their
order in the original deck. %It is clear that
%the definition oriented $1$-random-to-top shuffle agrees
%There are $\binom{n}{k} 2^k$
%oriented $k$-shuffles.

%
%{\bf The inverse permutation is needed for the next lemma. It is performed
%by taking the top $k$ cards, flipping the bottom $j$ as a block, and then
%inserting the blocks of $k-j$ cards and $j$ flipped cards into the deck,
%preserving the order within each block.}

For $k < n$, let
$\Xi^\dagger_k = \Xi^\dagger_{\{\alpha_k\}} \in \D(\BSym_n)$
and let $\Xi^\dagger_n = \Xi^\dagger_{\{\beta\}} \in \D(\BSym_n)$.
%e_1 - e_2, \ldots, e_k-e_{k+1} \}. \]
The following lemma is proved by extending the argument used in the proof of
Lemma~\ref{lemma:XiB}.

\begin{lemma}\label{lemma:XiBgen}
Let $k \le n$.
Let $\sigma \in \BSym_n$. Then $\sigma$ is in the support of $\Xi^\dagger_k$ if and only
if $\sigma^{-1}$ is
an oriented $k$-shuffle.
\end{lemma}

\begin{proof}
Let $\sigma$ be in the support of $\Xi^\dagger_k$.
Lemma~\ref{lemma:interval} implies that either $1 \le 1 \sigma < \ldots < k \sigma$
or there exists a unique $j \in \{1,\ldots, k\}$
such that $1 \le 1 \sigma < \ldots < (j-1)\sigma$ and $j \sigma < \ldots < k\sigma \le -1$.
Moreover, if $k < n$, then since $\beta \sigma \in \Phi^+$, Lemma~\ref{lemma:interval} implies that
$1 \le (k+1) \sigma < \ldots < n \sigma$.

Thus in our card shuffling interpretation,
$\sigma$ takes the top $k$ cards, flips the bottom $j$ of these cards as a block, for some $j \in \{0,1,
\ldots, k\}$,
and then inserts the blocks of $k-j$ unflipped cards and $j$ flipped cards
into the deck, preserving the order within each block. These shuffles
are exactly the inverses of the oriented $k$-shuffles.
\end{proof}

Under the canonical epimorphism $\D(\BSym_n)
\rightarrow \Cl(\BSym_n)$, the image of $\Xi^\dagger_k$ is
the permutation character $\pi_k^\dagger$ of $\BSym_n$ acting on the cosets
of the parabolic subgroup $\Sym_k \times (C_2 \wr \Sym_{n-k})$
generated by the reflections in the simple roots other than $\alpha_k$ if $k < n$,
or other than $\beta$ if $k=n$.
A convenient model for this coset space
is the set $\Omega_k$ of $k$-subsets of the short roots $\{ \pm e_1, \ldots, \pm e_n \}$
that contain at most one element of each pair $\{ e_i, -e_i\}$.
The action of $\BSym_n$
on $\Omega_k$ is inherited from its action on $\langle e_1, \ldots, e_n\rangle$.
Using the notation for elements of $\BSym_n$ introduced immediately before
(4.1.12) in \cite{JK}, it is clear that a $k$-subset $A = \{s_j e_j : j \in J \}
\in \Omega_k$
is fixed by $(u_1, \ldots, u_n ; \tau) \in \BSym_n$ if and only if
(i)~$J$ is fixed by $\tau$ and (ii)~$u_j s_j = s_{j \tau}$ for each $j \in J$.

Let $P^\dagger(k)$ be the transition matrix of the Markov chain on
$\BSym_n$ in which each step is given by choosing one of the
$2^k\binom{n}{k}$ oriented
$k$-shuffles uniformly at random. Let $P^{\dagger'}(k)$ be the analogous transition matrix
when only non-identity shuffles are chosen.
The same argument used in \S\ref{sec:Phatarfod} now proves the following proposition.

\begin{proposition}\label{prop:eigenvaluesB} %{prop:nonidEigenvalues}
The eigenvalues of $P^\dagger(k)$
are $\pi^\dagger_k(\tau)/ 2^k\binom{n}{k}$
for $\tau \in \BSym_n$ and the eigenvalues of
$P^{\dagger'}(k)$ are $(\pi^\dagger_k(\tau)-1) / ( 2^k\binom{n}{k} - 1)$.\hfill$\qed$
\end{proposition}

The two extreme cases are worth noting. It is clear that $\pi^\dagger_1(\sigma)
= \pi^\dagger(\sigma)$ for $\sigma \in \BSym_n$; the model $\Omega_1$ for the coset space
shows that $\pi^\dagger\bigl( (u_1,\ldots, u_n ; \tau) \bigr)$ is the number of
elements $e_i$ such that $u_i=1$ and $i\tau = i$.
An $n$-shuffle is the inverse of
the shuffle performed by separating the deck into two parts, flipping the part
containing the bottom card, and then riffle-shuffling the two parts.
By (i) and (ii) we see that $\pi_n^\dagger\bigl( (u_1,\ldots, u_n ; \tau) \bigr)$
is the number of sequences $(s_1,\ldots,s_n)$ such that $s_j \in \{+1,-1\}$
and $u_j s_j = s_{j \tau}$
for each $j$. Hence the corresponding
eigenvalue is $2^d$, where~$d$ is the number of orbits of $\tau$ on $\{1,\ldots, n\}$.

\section{Cheating random-to-top shuffles: Type D}\label{sec:typeD}

\subsection{Type D shuffles and partition moves}
For $n \in \N$ we define the Coxeter group $\DSym_n$ of Type D by
\[ \DSym_n = \bigl\{ \sigma \in \BSym_n : \text{$\bigl|\{1,\ldots, n\} \sigma \cap \{-1,\ldots, -n\}\bigr|$
is even} \bigr\}. \]
%Since $\DSym_n$ has index $2$ in $\BSym_n$ and $(-n,n) \not\in \DSym_n$, given
%any $\tau \in \BSym_n$, exactly one of $\tau$ and $\tau(-n,n)$ is in $\DSym_n$.
The shuffles and partition moves relevant to this case are defined as follows.
We shall use the superscript $\ddagger$ to denote quantities relevant to this type.
%Note that if $\sigma \in \BSym_n$ then either $\sigma$ or $\sigma (-n,n)$ is in $\DSym_n$,
%but not both.

\begin{definition}
Let $n \in \N$ and let $k \le n$.
Let $\sigma \in \DSym_n$.
Then $\sigma$ is a \emph{cheating $k$-shuffle}
if and only if either $\sigma$ or $\sigma(-n,n)$ is a $k$-shuffle.
A cheating $1$-shuffle
will be called a \emph{cheating random-to-top shuffle}.
Let~$\SD_t(n)$ be the number of sequences of~$t$ cheating random-to-top shuffles
whose product is $\id_{\DSym_n}$. Define $\SDp_t(n)$ analogously, excluding the
identity permutation $\rho_1$.
\end{definition}

Thus the cheating $k$-shuffles are obtained
from the oriented $k$-shuffles by composing with
a `cheating' flip of the bottom card, whenever this is necessary to
arrive in the group $\DSym_n$. It easily follows that the cheating
random-to-top shuffles are
%Thus the cheating $k$-shuffle are in canonical
%bijection with the oriented $k$-
%shuffles in Type B.
%; they are obtained by composing with a
% `cheating' flip of the bottom card, whenever
%this is necessary to arrive in the group $\DSym_n$.
\[ \{ \rho_m : 1 \le m \le n \} \cup \{ \bar{\rho}_m (-n,n) : 1 \le m \le n \} \]
where $\rho_m$ and $\bar{\rho}_m$ are as defined in
\S\ref{sec:typeBstatement}.
Note that a cheating random-to-top shuffle either flips no cards, or flips both the lifted card and
the (new, in the case of $\bar{\rho}_n$) bottom card.

\begin{definition}
Let $t$, $n \in \N_0$.
Let $\MD_t(n)$ be the number of sequences of $t$ double-moves that start at $((n), \varnothing)$
and finish at either $((n), \varnothing)$ or $(\varnothing, (n))$. Define $\MDp_t(n)$
analogously, considering only non-exceptional double-moves.
\end{definition}

%Maybe can find reference for shuffles that manipulate base of pack used
%in sleights / crooked deals to justify name.
 As in the case of Type B, the sequences of shuffles counted by $\SD_t(n)$ and $\SDp_t(n)$ are
in bijection with certain marked set partitions of $\{1,\ldots, t\}$ into at most $n$ parts, but
it now seems impossible to give a more direct description of these partitions.
Our strongest analogue of Theorem~\ref{thm:main} is as follows.

\begin{theorem}\label{thm:mainD}
For all $t \in \N_0$ and $n \in \N$ such that $n\ge 2$ we have $\SD_t(n) = \MD_t(n)$ and $\SDp_t(n) = \MDp_t(n)$.
\end{theorem}

We remark that since $\DSym_1$ is the trivial group we have $\SD_t(1) = 1$
whereas $\MD_t(1) = 2^t$ for all $t \in \N_0$.
Similarly  $\SDp_0(1) = 1$ and $\SDp_t(1) = 0$ for all $t \in \N$,
whereas $\MDp_t(1) = 1$ for all $t \in \N_0$.

%The case $n=2$ has to be treated separately in the proof: given the connection with the type D root system,
%this is not surprising.

%Note this fails for n=2: then the only shuffle sending e_1-e_2 to a negative root (- itself)
% and sending e_1+e_2 to a positive root is (1,2)(-1,-2), so Delta = (1,2)(-1,-2) and
%Xi = 1 + (1,2)(-1,-2), not in agreement with proof.

\subsection{Proof of Theorem~\ref{thm:mainD}}
We begin with the character theoretic part
of the proof.
%Let $\DSym_{n-1} = \BSym_{n-1} \cap \DSym_n$,
Let $\pi^\ddagger = \pi^\dagger \res_{\DSym_n}$
%1_{\DSym_{n-1}}\ind^{\DSym_n}$
and  let $\theta^\ddagger = \pi^\ddagger - 1_{\DSym_n}$.

%Fix $n \in \N$. For $m \in \{1,\ldots, n\}$ let
%\[  \]

\begin{lemma}\label{lemma:branchingD}
Let $t \in \N_0$ and let $n \in \N$. % be such that $n \ge 2$.
Then $\MD_t(n) = \langle \pi^{\ddagger t}, 1_{\DSym_n} \rangle$
and $\MDp_t(n) = \langle  \theta^{\ddagger t}, 1_{\DSym_n} \rangle$.
\end{lemma}

\begin{proof}
Since $1_{\DSym_n}\!\!\ind^{\BSym_n} = \chi^{((n), \varnothing)} + \chi^{(\varnothing, (n))}$,
it follows from Frobenius reciprocity that
\[
\langle \pi^{\ddagger t}, 1_{\DSym_n} \rangle
= \langle \pi^{\dagger t}\Res_{\DSym_n}, 1_{\DSym_n} \rangle \\
= \langle \pi^{\dagger t}, \chi^{((n) \sep \varnothing)} + \chi^{(\varnothing \sep (n))} \rangle.
\]
The first part of the lemma now follows from Lemma~\ref{lemma:branchingB}. %${}^\dagger$.
The second part is proved analogously.
\end{proof}

%
%Since $\widetilde{\sgn_m}\res_{\DSym_n} = 1_{\DSym_n}$
%we have $\chi^{(\lambda \sep \lambda^\star)}\res_{\DSym_n}=
%\chi^{(\lambda^\star \sep \lambda)}\res_{\DSym_n}$.
%It is
%Basic Clifford theory now implies that the
%irreducible characters of $\DSym_n$ are precisely the restrictions to $\DSym_n$ of the
%irreducible characters of $\BSym_n$, with the exception that if $n$ is even and
%$\lambda$ is a partition of $n$ then
%$\chi^{(\lambda, \lambda)}$
%splits as the sum of two irreducible characters of $\DSym_n$.
%The restrictions of $\chi^{(\lambda, \lambda^\star)
%%This splitting is irrelevant to us
%%because $\pi^\ddagger$ and $\theta^\ddagger$ extend to characters of $\BSym_n$.
%Therefore $1_{\DSym_n} = \chi^{((n), \varnothing)} = \chi^{(\varnothing, (n))}$
%and the lemma follows from Lemma~\ref{lemma:branching}$^\dagger$.
%\end{proof}

A root system for $\DSym_n$, constructed inside the span of the root system for $\BSym_n$ already defined,
has positive roots $\Psi^+ = \{ e_i-e_j : 1 \le j \le n \} \cup \{e_i + e_j : 1 \le j \le n \}$,
and simple roots $\alpha_1, \ldots, \alpha_{n-1}$ and
$\gamma$, where $\gamma = e_{n-1} + e_n$.
Let $\D(\DSym_n)$ be the descent subalgebra of the rational group algebra $\Q\!\DSym_n$,
as defined in \cite[Theorem~1]{Solomon}, for this choice of simple roots. Let $\Xi^\ddagger_I$
for $I \subseteq \{ \alpha_1, \ldots, \alpha_{n-1}, \gamma \}$ be the canonical basis elements,
defined in the same way as for Type B.

Let $\Xi^\ddagger = \sum_{m=1}^n (\rho_m^{-1} + (-n,n)\bar{\rho}_m^{-1})$ and let
$\Delta^\ddagger = \Xi^\ddagger - \mathrm{id}_{\DSym_n}$.
The following lemma is the analogue of Lemma~\ref{lemma:XiB} for Type D. %${}^\dagger$.

\begin{lemma}\label{lemma:XiD}
Let $n \in \N$ be such that $n \ge 3$. Then $\Xi^\ddagger$, $\Delta^\ddagger \in \D(\DSym_n)$ and under the canonical algebra epimorphism
$\D(\DSym_n) \rightarrow \Cl(\DSym_n)$ we have $\Xi^\ddagger \mapsto \pi^\ddagger$ and
$\Delta^\ddagger \mapsto \pi^\ddagger - 1_{\Sym_n}$.
\end{lemma}

\begin{proof}
Let $\sigma \in \DSym_n$.
Suppose that $\alpha_1$ is the unique simple root~$\delta$ such
that $\delta\sigma \not\in \Psi^+$. Observe that if $n \sigma \in \{-1,\ldots, -n\}$
then, since $\gamma \sigma \in \Psi^+$, we have $(n-1)\sigma \in \{1,\ldots, n\}$
and $(n-1)\sigma < |n\sigma|$.
It therefore follows from Lemma~\ref{lemma:interval} that
$\{2,\ldots, n-1\}\sigma \subseteq \{1,\ldots, n\}$.
Since $\bigl| \{1,\ldots,n\}
\sigma \cap \{-1,\ldots,-n\} \bigr|$ is even, %, by definition of $\DSym_n$,
either $\{1,n\}\sigma \subseteq \{1,\ldots, n\}$ or $\{1,n\}\sigma \subseteq \{-1,\ldots, n\}$.
Hence \emph{either}
\begin{itemize}
\item[(1)] $1 \sigma, 2 \sigma
\in \{1,\ldots, n\}$ and $1\sigma > 2 \sigma < \ldots < n \sigma$, \emph{or}
\item[(2)] $1 \sigma, n \sigma \in \{-1,\ldots, -n\}$ and
\hbox{$2 \sigma < \ldots < (n-1)\sigma < |n \sigma|$}.
%In (B) have (n-1)\sigma = (n-1).
\end{itemize}
The permutations $\sigma$ in Case (1) are
$\{ \rho_m^{-1} : 1 \le m \le n \}$.
In Case (2) the chain of inequalities implies that $|n \sigma| \ge n-1$ and $(n-1)\sigma \ge n-2$.
If $n\sigma = -(n-1)$
then $(n-1)\sigma = n-2$ and the unique permutation is
\[ (-n, n) \bar{\rho}_n^{-1} = (n, -(n-1), \ldots -1)(-n, (n-1), \ldots, 1). \]
The permutations such that $n \sigma = -n$ are $(-n,n) \bar{\rho}_m^{-1}$
for $1 \le m \le n-1$. Hence $\Xi^\ddagger = \Xi_{\{\alpha_1\}} \in \DSym_n$.

It is easily seen that $\pi^\ddagger = 1_{\DSym_{n-1}}\hskip-2pt\ind^{\DSym_n}$, and
that $\DSym_{n-1}$ is generated by the reflections in the simple roots $\alpha_i$ for
$2 \le i \le n-1$
and~$\gamma$. Therefore, by \cite[Theorem~1]{Solomon}, under the
canonical algebra epimorphism $\D(\DSym_n) \rightarrow \Cl(\DSym_n)$
we have $\Xi^\ddagger \mapsto \pi^\ddagger$. This completes the proof.
\end{proof}

%Following
%the usual strategy, we must show that if $n \ge 2$
%then $\Xi^\ddagger, \Delta^\ddagger \in \D(\DSym_n)$
%and the coefficients of $\id_{\DSym_n}$ in
%$\Xi^{\ddagger t}$ and $\Delta^{\ddagger t}$ are $\MD_t(n)$ and $\MDp_t(n)$, respectively.

We are now ready to prove Theorem~\ref{thm:mainD}.
When $n \ge 3$ the theorem follows from Lemma~\ref{lemma:branchingD} and
Lemma~\ref{lemma:XiD} by the argument used to prove Lemma~\ref{lemma:daggerSM}.
In the remaining case $\DSym_2
= \langle (-1,-2)(1,2), (-1,1)(-2,2) \rangle$ is the Klein $4$-group
and every permutation in $\DSym_2$ is a cheating
top-to-random shuffle. Moreover,~$\pi$ is the sum of all irreducible characters
of $\DSym_2$ and $\theta$ is the sum of all non-trivial irreducible characters of $\DSym_2$.
%(Thus both~$\pi$ and $\theta$ contain the
%two irreducible summands of $\chi^{((1),(1))}\!\res_{\DSym_2}$.)
Since $\DSym_2$ is abelian, the group algebra $\Q\!\DSym_2$ is isomorphic to the
algebra of class functions $\Cl(\DSym_2)$. Hence the coefficient of $\id_{\DSym_n}$ in
$\Xi^{\ddagger t}$ is  $\langle \pi^{\ddagger t}, 1_{\DSym_n} \rangle$,
and  the coefficient
of $\id_{\DSym_n}$ in $\Delta^{\ddagger t}$ is $\theta^{\ddagger t}$. The theorem
now follows from Lemma~\ref{lemma:branchingD}.

\subsection{Eigenvalues of the  cheating oriented $k$-shuffle}

For $k < n$ let $\Xi^\ddagger_k = \Xi^\ddagger_{\{\alpha_k\}}$
and let $\Xi^\ddagger_n = \Xi^\ddagger_{\{\gamma\}}$.
The following lemma is the analogue of Lemma~\ref{lemma:XiBgen}
and generalizes Lemma~\ref{lemma:XiD}.
Let $\mathcal{S}(\Gamma)$ denote the support of $\Gamma \in \Q\!\BSym_n$.

\begin{lemma}\label{lemma:XiDgen}
Let $k \le n$. Let $\sigma \in \DSym_n$. Then $\sigma$ is in the support of $\Xi^\ddagger_k$ if and only
if $\sigma^{-1}$ is
a cheating $k$-shuffle.
\end{lemma}

\begin{proof}
 Let $\sigma \in \BSym_n$. Suppose first of all that $k < n$.
It is clear that $\sigma \in \mathcal{S}(\Xi^\ddagger_k)$ and $n\sigma \in \{1,\ldots, n\}$
if and only if $\sigma \in \mathcal{S}(\Xi^\dagger_k) \cap \DSym_n$. By Lemma~\ref{lemma:XiBgen},
this is the case if and only if $\sigma^{-1}$ is a cheating $k$-shuffle
that leaves the bottom card unflipped.
If $n\sigma \in \{-1,\ldots,-n\}$ and $\gamma \sigma \in \Psi^+$, then, as seen
in the proof of Lemma~\ref{lemma:XiD}, we have
$(n-1)\sigma \in \{1,\ldots, n\}$ and $(n-1)\sigma < |n \sigma|$.
Hence if $\tau = (-n,n)\sigma \in \BSym_n$ then $n\tau \in \{1,\ldots,n\}$ and
$(n-1)\tau < n\tau$. It follows that $\sigma \in \mathcal{S}(\Xi^\ddagger_k)$ and $n\sigma \in \{-1,
\ldots,-n\}$ if and only if $(-n,n)\sigma \in \mathcal{S}(\Xi^\dagger_k) \cap \DSym_n$.
By Lemma~\ref{lemma:XiBgen}, this is the case if and only if $\sigma^{-1}$ is a cheating
$k$-shuffle that flips the bottom card.

Since $\sigma \in \mathcal{S}(\Xi^\ddagger_n)$ if and only if $\alpha_j \sigma \in \Psi^+$
for $1\le j \le n-1$, which is the case if and only if $\alpha_j \sigma \in \Phi^+$ for
$1 \le j \le n-1$, we have $\mathcal{S}(\Xi^\ddagger_n) = \mathcal{S}(\Xi^\dagger_n) \cap \DSym_n$.
It now follows from Lemma~\ref{lemma:XiBgen}
that $\sigma \in \mathcal{S}(\Xi^\ddagger_n)$ if and only if $\sigma^{-1}$
is a cheating $k$-shuffle that leaves the bottom card unflipped.
\end{proof}

It is worth noting that a small extension of this argument shows
that $\mathcal{S}(\Xi^\ddagger_k) = \mathcal{S}(\Xi^\dagger_k) \cap \DSym_n$
(and so no permutations in $\mathcal{S}(\Xi^\ddagger_k)$ involve the cheating flip of the bottom
card) if and only
if $k = n-1$ or $k=n$.

%We note that the proof shows that none of the shuffles
%in the support of $\Xi^\ddagger_n$ involve the cheating flip of the bottom card.
%while a shuffle $\sigma$ in the support of $\Xi^\ddagger_k$ for $k < n$ involves a cheating flip
%if and only if $n\sigma \in \{-1,\ldots,-n\}$.

Under the canonical epimorphism $\D(\DSym_n) \rightarrow \Cl(\DSym_n)$, the image of
$\Xi_k^\ddagger$ is the permutation character $\pi_k^\ddagger$ of $\DSym_n$ acting
on the cosets of the parabolic subgroup $\bigl(\Sym_k \times (C_2 \wr \Sym_{n-k})\bigr) \cap \DSym_n$.
Thus $\pi_k^\ddagger = \pi_k^\dagger\res_{\DSym_n}$.

Let $P^\ddagger(k)$ be the transition matrix of the Markov chain on
$\DSym_n$ in which each step is given by choosing one of the
$2^k\binom{n}{k}$ cheating $k$-shuffles uniformly at random. Let $P^{\ddagger'}(k)$ be the analogous chain
where only non-identity shuffles are chosen.
The same argument used in \S\ref{sec:Phatarfod} now proves the following proposition.

\begin{proposition}\label{prop:eigenvaluesD} %{prop:nonidEigenvalues}
The eigenvalues of $P^\ddagger(k)$
are $\pi^\ddagger_k(\tau)/ 2^k\binom{n}{k}$
for $\tau \in \DSym_n$ and the eigenvalues of
$P^{\ddagger'}(k)$ are $(\pi^\ddagger_k(\tau)-1) / ( 2^k\binom{n}{k} - 1)$.\hfill$\qed$
\end{proposition}

In particular, we observe
that the eigenvalues of the Type D shuffle are exactly the eigenvalues of the Type B shuffle
coming from elements of $\DSym_n \le \BSym_n$.

\section{Generating functions and asymptotics}

For $t \in \N_0$ and $n \in \N$,
the Stirling number of the second kind $\stir{t}{n}$ satisfies $\stir{t}{n} = B_t(n)-
B_{t}(n-1)$. By analogy we define $\stir{t}{n}' = B'_t(n) - B'_{t}(n-1)$,
$\stir{t}{n}^\dagger = B^\dagger_t(n) - B^\dagger_t(n-1)$ and
$\stir{t}{n}^{\dagger\prime} = B^{\dagger\prime}_t(n) - B^{\dagger\prime}_t(n-1)$.
%%Similarly we define
%$\stir{t}{n}^\ddagger = M^\dagger_t(n) - M^\dagger_t(n-1)$
%and
%$\stir{t}{n}^{\ddagger '} = M^{\ddagger '}_t(n) - M^{\ddagger '}_t(n-1)$.
When $n=0$ each generalized Stirling number is defined to be $1$ if $t=0$ and otherwise $0$.

%Define $\stir{t}{n}^\ddagger$ and $\stir{t}{n}^{\ddagger '}$ analogously, using
%$M^\ddagger_t(n) = S^\ddagger_t(n)$.
Theorem~\ref{thm:main} and Theorem~\ref{thm:mainB} give several
combinatorial interpretations of these numbers. In particular, we note
that $\stir{t}{n}$ is the number of sequences of $t$ random-to-top shuffles
that leave a deck of $n$ cards invariant while lifting every card at least once, and
$\stir{t}{n}^\dagger$ is the number of such sequences of oriented random-to-top shuffles.
Moreover $\stir{t}{n}'$ and $\stir{t}{n}^{\dagger\prime}$ have similar interpretations,
considering only sequences of non-identity shuffles.

In this section we give generating functions and asymptotic results on the generalized Bell and Stirling
Numbers. Along the way we shall see a number of relationships between these numbers.
Some of these results are obtained using basic arguments from residue calculus: we refer
the reader to \cite[Section 5.2]{Wilf} for an account of this method.

%We begin by observing that, since
%there are $2^{n-t}$ ways to mark the elements of a set partition of $\{1,\ldots, t\}$
%into exactly $n$ parts so that an even number of elements in each set are marked, we have
%$\stir{t}{n}^\dagger = 2^{n-t} \stir{t}{n}$.

\subsection{Relating $B_t(n)$ to $B'_t(n)$ and the asymptotics of $B'_t(t)$}

The numbers $B'_t(t)$ are considered by Bernhart \cite{Bernhart}, who describes the associated set partitions
as \emph{cyclically spaced}. The following lemma generalizes a result
in \S 3.5 of \cite{Bernhart}. The bijective proof given therein also generalizes,
but we give instead a short algebraic proof as an application of Theorem~\ref{thm:main} and
Lemma~\ref{lemma:branching}.

\begin{lemma}\label{lemma:Bernhart}
If $t$, $n \in \N$ then $B'_t(n) + B'_{t-1}(n) = B_{t-1}(n-1)$.
%If $n$, $t \in \N$ then $\displaystyle \stir{t}{n}' + \stir{t-1}{n}' = \stir{t-1}{n-1}$.
\end{lemma}

\begin{proof}
By Theorem~\ref{thm:main} it is equivalent to prove that
$M'_t(n) + M'_{t-1}(n) = M_{t-1}(n-1)$. Let $\pi$ be the natural permutation
character of $\Sym_n$ and let
$\theta = \chi^{(n-1,1)}$, as in Lemma~\ref{lemma:branching}.
Note that $\theta\!\res_{\Sym_{n-1}}$ is the natural permutation character of $\Sym_{n-1}$.
Since restriction commutes with taking products of characters, it
follows from Lemma~\ref{lemma:branching}  that
$M_{t-1}(n-1)  = \langle \theta^{t-1}\hspace*{-3pt}\res_{\hspace*{1pt}\Sym_{n-1}}, 1_{\Sym_{n-1}} \rangle$. By Frobenius reciprocity, $ \langle \theta^{t-1}\hspace*{-3pt}\res_{\hspace*{1pt}\Sym_{n-1}}, 1_{\Sym_{n-1}} \rangle
= \langle \theta^{t-1}, \pi \rangle$. Now
\[ \langle \theta^{t-1}, \pi \rangle = \langle \theta^{t-1}, \theta \rangle +
\langle \theta^{t-1}, 1_{\Sym_n} \rangle = \langle \theta^t, 1_{\Sym_n} \rangle
+ \langle \theta^{t-1}, 1_{\Sym_n} \rangle \]
which is equal to $M'_t(n) + M'_{t-1}(n)$, again by Lemma~\ref{lemma:branching}.
%
%\begin{align*}
%B_{t-1}(n-1) &= \langle \theta^{t-1}\Res_{\Sym_{n-1}}, 1_{\Sym_{n-1}} \rangle \\
%             &= \langle \theta^{t-1}, \pi \rangle \\
%             &= \langle \theta^{t-1}, \theta \rangle +  \langle \theta^{t-1}, 1_{\Sym_{n}} \rangle \\
%             &= \langle \theta^{t}, 1_{\Sym_n} \rangle + \langle \theta^{t-1}, 1_{\Sym_{n}} \rangle
%             &=  B'_t(n) + B'_{t-1}(n)
%\end{align*}
%as required.
\end{proof}

We note that the quantity $B_t'(n) + B_{t-1}'(n)$  appearing in Lemma~\ref{lemma:Bernhart}  has a natural interpretation: it counts set partitions which are \emph{spaced}, but not necessarily \emph{cyclically spaced}. To prove this we use the associated Stirling numbers:
for $t \in \N$ and $n \in \N_0$ define $\stir{t}{n}^{\star} = \stir{t}{n}' + \stir{t-1}{n}'$.

\begin{proposition}\label{prop:star}
For $t \in \N$ and $n\in \N_0$, $\stir{t}{n}^\star$ is equal to the number of set partitions of $\{1,\ldots, t\}$ into $n$ parts, such that $i$ and $i+1$ are not in the same part for any $i$.
\end{proposition}

\begin{proof}
The set partitions of $\{1,\ldots, t\}$ into exactly $n$ sets
counted by $\stir{t}{n}^\star$ but not by $\stir{t}{n}'$ are those with $1$ and $t$ in the same part. But in this case $t-1$ cannot also be in the same part as $1$, and so deleting $t$ gives a bijection between these extra set partitions and the set partitions counted by $\stir{t-1}{n}'$.
\end{proof}
%We note that Corollary~\ref{cor:Bernhart2} was first proved bijectively in
%\cite[\S 3.5]{Bernhart}, using the relation  $B'_t(t) + B'_{t-1}(t-1) =
%B_{t-1}(t-1)$ implied by Lemma~\ref{lemma:Bernhart}.
%
%

We now use Lemma~\ref{lemma:Bernhart}
to get the exponential generating function for the~$B'_t(t)$.

\begin{proposition}\label{prop:Bernhart}
\[ \sum_{t=0}^\infty \frac{B'_t(t)}{t!} x^t = \exp\bigl( \exp(x) - 1 - x \bigr). \]
\end{proposition}

\begin{proof}
Taking $t \in \N$ and $n \ge t$, Lemma~\ref{lemma:Bernhart} gives
$B'_t(t) +B'_{t-1}(t-1) = B_{t-1}(t-1)$. It follows that if
$F(x) = \sum_{t=0}^\infty B'_t(t)x^t/t!$ is the exponential
generating function for $B'_t(t)$ then
\[ F(x) + F'(x) = \sum_{t=0}^\infty \frac{B_t(t)}{t!}x^t = \exp(\exp(x) - 1). \]
Since $\exp(\exp(x)-1-x)$ solves this differential equation, and agrees
with~$F$ when $x=0$, we
have $F(x) = \exp(\exp(x) -1 -x)$.
\end{proof}

Thus $\sum_{t=0}^\infty B'_t(t) x^t/t! = \exp(-x) \sum_{t=0}^\infty B_t(t) x^t/t!$.
As an immediate corollary we obtain  $B'_t(t) = \sum_{s=0}^t \binom{t}{s} (-1)^{t-s} B_s(s)$
and $B_t(t) = \sum_{s=0}^t \binom{t}{s} B'_s(s)$; these formulae are related by binomial inversion.
Using this formula for $B'_t(t)$
and the standard result
\begin{equation} \label{eq:BExp} B_t(t) = \frac{1}{\mathrm{e}} \sum_{j=0}^\infty
\frac{j^t}{j!} \end{equation}
(see for example \cite[Equation (1.41)]{Wilf}, or
sum Equation~\eqref{eq:stirling} below over all $n \in \N_0$)
we obtain %NB this is 1.6.10 in the second edition
\begin{equation} \label{eq:BprimedExp} B'_t(t) =
%\frac{1}{\mathrm{e}} \Bigl(
%(-1)^t + \sum_{k=1}^\infty \frac{(k-1)^t}{k!} \Bigr). \]
%\frac{(-1)^t}{\mathrm{e}} +
\frac{1}{\mathrm{e}} \sum_{j=0}^\infty \frac{(j-1)^t}{j!}. \end{equation}
This equation appears to offer the easiest route to the asymptotics of $B_t(t)'$.
Let $W(t)$ denote Lambert's $W$ function,
defined for $x \in \R^{\ge 0}$ by the equation \hbox{$W(x)\mathrm{e}^{W(x)} = x$}.

\begin{corollary}\label{cor:Bernhart}
We have
\[ \frac{B'_t(t)}{B_t(t)} \sim \frac{W(t)}{t} \text{ as $t \rightarrow \infty$.} \]
\end{corollary}

\begin{proof}
Let $m(t) = \lfloor \exp W(t-1/2) \rfloor$. The solution to Exercise~9.46 in \cite{CMath}
can easily be adapted to show
that
\[ B_t(t)' = \e^{m(t)-t-1/2}m(t)^{t-1} \sqrt{\frac{m(t)}{m(t)+t}}
\bigl( 1 + \O(t^{-1/2} \log t) \bigr).\]
Comparing with the analogous formula for $B_t(t)$ proved in \cite{CMath} we obtain
$B'_t(t)/B_t(t) \sim 1/m(t)$ as $t \rightarrow \infty$.
We now use the fact that
\[ m(t) = \exp W(t-1/2) = \frac{t-1/2}{W(t-1/2)} \sim \frac{t}{W(t)} \text{ as $t \rightarrow \infty$}, \]
where the final asymptotic equality follows from the elementary bounds $\log t - \log\log t \le W(t) \le \log t$.
\end{proof}

%\begin{proof}
%The first displayed equation in the solution to Exercise 9.46 in \cite{CMath} implies
%that if $|k| \le m^{-1/2+\epsilon}$ then
%\[ \log \frac{(m+k)^n}{(m+k)!} = n \log m - m \log m + m - \mfrac{1}{2}\log 2\pi m - \frac{(m+n)k^2}{2m^2}
%+ \O(m^{-1/2+\epsilon} \log n. \]
%Hence
%\[ \frac{(m+k-1)^n}{(m+k)!} = \frac{\e^m m^{n-m-1}}{\sqrt{2 \pi m}} \exp \bigl( - \frac{(m+n)k^2}{2m^2}
%\bigr) \bigl( 1 + \O(m^{-1/2+\epsilon}\log n) \bigr). \]
%Again following the argument in \cite{CMath}, using
% $\sum_{k \in \Z} \exp (-\alpha k^2) \sim 1/\sqrt{2\alpha}$ as $\alpha
%\rightarrow \infty$ and $\sum_{k \ge a} \exp(-\alpha k^2) = \O(\exp(-a))$ to bound the tail sums,
%we obtain
%\[ B_t(t)' = \e^{m-n-1/2}m^{n-1} \sqrt{\frac{m}{m+n}} \bigl( 1 + \O( n^{-1/2+\epsilon} \log n \bigr).\]
%NB Get rid of the epsilon by estimating m = o(n).
%Comparing with the analogous formula for $B_t(t)$ proved in \cite{CMath} we get
%$\frac{B_t(t)'}{B_t(t)} \sim \frac{1}{m}$.
%The result now follows since $m = \exp W(t-1/2) = (t-1/2)/W(t-1/2)$ and
%$W(t-1/2)/W(t) \sim 1$ as $t \rightarrow \infty$.
%\end{proof}

%More precise asymptotic results on $B_t(t)'$ can be proved by adapting the method using
%in \cite{\S 6.2}[deBruijn], but we shall not pursue this possibility here.
We also obtain the following result,
 which is proved bijectively in \cite[\S 3.5]{Bernhart}.

\begin{corollary}\label{cor:Bernhart2}
If $t\in\N_0$ then $B'_t(t)$ is the number of set partitions of $\{1,\ldots,t\}$ into parts
of size at least two.
\end{corollary}

\begin{proof}
Since $\exp(x)-1-x$ is the exponential
generating function enumerating sets of size at least two, the corollary
follows from Proposition~\ref{prop:Bernhart} using \cite[Theorem~3.11]{Wilf}.
\end{proof}

It is worth noting that this corollary does not extend to the Stirling numbers $\stir{t}{n}'$.
For example, if $t \ge 2$ then $\stir{t}{1}' = 0$, whereas the unique
set partition of $\{1,\ldots, t\}$ into a single part obviously has all parts of size at least $2$.

\subsection{Generating functions and asymptotics for $\stir{t}{n}'$ and $B'_t(n)$}

Let $t$, $n \in \N$. By definition we have $\stir{t}{n}' = B'_t(n) - B'_{t}(n-1)$. Provided
$n \ge 2$, Lemma~\ref{lemma:Bernhart} applies to both summands, and we obtain
$\stir{t}{n}' + \stir{t-1}{n}' = \stir{t-1}{n-1}$.
%When $n=1$ this equation still
%holds, except when $t=1$, in which case $\stir{1}{1}' = 0$, $\stir{0}{1}' = 0$, but
%$\stir{0}{0}' = 1$.
%\begin{align*} \stir{t}{n}' + \stir{t-1}{n}' &= \bigl( B'_t(n)-B'_t(n-1) \bigr) +
%\bigl( B'_{t-1}(n)-B'_{t-1}(n-1) \bigr) \\ &= \bigl( B'_t(n) + B'_{t-1}(n) \bigr)
%- \bigl( B'_t(n-1) + B'_{t-1}(n-1) \bigr) \\ &= B_{t-1}(n-1) - B_{t-1}(n-2) \\ &= \stir{t-1}{n-1}'.
%\end{align*}
From the ordinary generating function
$\sum_{t=0}^\infty \stir{t}{n} x^t = x^n \prod_{j=1}^n 1/(1-jx)$
 (see for example \cite[(1.36)]{Wilf}) we now get
\begin{equation}\label{eq:gfstir'tn} \sum_{t=0}^\infty \stir{t}{n}' x^t = \frac{x^{n}}{1+x} \prod_{j=1}^{n-1} \frac{1}{1-jx}
\end{equation}
for $n \ge 2$. (When $n=1$ we have $\stir{t}{1}' = 0$ for all $t \in \N$, so the generating function
is zero.)
A simple residue calculation now shows that provided $n \ge 3$, we have
\[ \stir{t}{n}' \sim \frac{(n-1)^{t}}{n!} \quad\text{as $t \rightarrow \infty$}.\]
(When $n=2$ we have $\stir{t}{2}' = 1$ if $t$ is even, and $\stir{t}{2}' = 0$ if $t$ is odd.)
It easily follows that same asymptotic relation holds for $B'_t(n)$. Thus if $n \ge 3$ then
$\frac{1}{n}\stir{t}{n-1}$, $\stir{t}{n}'$, $\frac{1}{n}B_t(n-1)$
and $B'_t(n)$ are all asymptotically equal to
$(n-1)^t/n!$ as $t \rightarrow \infty$. Moreover, by calculating all residues
in~\eqref{eq:gfstir'tn} we obtain
the explicit formula
\begin{equation}
\label{eq:stirlingprime}  %NB j gets switched to n-j = k in a calc so reason not to use j
\stir{t}{n}' = \frac{1}{n!}\sum_{k=0}^{n-1} (-1)^k\binom{n}{k} (n-k-1)^t + \frac{(-1)^{n+t}}{n!}
\end{equation}
valid for $n \ge 2$.
This formula may be compared with the well known identity
\begin{equation} \label{eq:stirling}
\stir{t}{n} = \frac{1}{n!}\sum_{k=0}^{n}(-1)^k \binom{n}{k} (n-k)^t ,
 \end{equation}
valid for all $t$, $n \in \N_0$,
which has a short direct proof using the Principle of Inclusion and
Exclusion. When $n\ge 3$, an easy corollary of \eqref{eq:gfstir'tn} is the
recurrence $\stir{t}{n}' = \stir{t-1}{n-1}' + (n-1)\stir{t-1}{n}'$,
analogous
to the well known $\stir{t}{n} = \stir{t-1}{n-1} + n\stir{t-1}{n}$.
However, while the recurrence for $\stir{t}{n}$ has a very simple bijective proof,
the authors know of no such proof for the recurrence for $\stir{t}{n}'$.

% Does the formula for $\stir{t}{n}'$?

%Further recurrence?

\subsection{Generating function and asymptotics of $B^\dagger_t(t)$}
%\subsection{Relating $B_t(n)$ to $B'_t(n)$ and the asymptotics of $B^\dagger_t(t)$ and $B^{\dagger '}_t(t)$}

The exponential generating function enumerating non-empty sets with an even number of their
elements marked is $\mfrac{1}{2}(\exp 2x - 1)$. Hence, by \cite[Theorem~3.11]{Wilf},
we have
\[ \sum_{n=0}^\infty \frac{B^\dagger_t(t)}{t!}x^t = \exp \bigl( \mfrac{1}{2} (\exp 2x - 1) \bigr). \]
Since there are $2^{t-n}$ ways to mark the elements of a set partition of $\{1,\ldots, t\}$
into $n$ parts so that an even number of elements in each part are marked, we have $\stir{t}{n}^\dagger
= 2^{t-n}\stir{t}{n}$. Using this, a routine adaption of the
proof of \cite[Equation 1.41]{Wilf} shows that
\begin{equation}
\label{eq:BdaggerExp} B^\dagger_t(t) =\frac{1}{\sqrt{e}} %\e^{-1/2}
\sum_{j=0}^\infty \frac{(2j)^t}{2^jj!}.
\end{equation}
The method used to prove Corollary~\ref{cor:Bernhart} then shows that
\[ B^\dagger_t(t) = 2^t \e^{\ell(t)-t}\ell(t)^{t} \sqrt{\frac{\ell(t)}{\ell(t)+t}}
\bigl( 1 + \O( t^{-1/2} \log t) \bigr) \]
where $\ell(t)$ is defined by the equation $(\log \ell(t) + \log 2)\ell(t) = t- 1/2$.

%Similar character theoretic arguments to those in Section~6.1 give numerical relationships in Types B and D.

\subsection{Relating $B^\dagger_t(n)$ to $B^{\dagger\prime}_t(n)$ and the asymptotics
of $B_t^{\dagger\prime}(n)$}

Lemma~\ref{lemma:Bernhart} has the following analogue in Type~B.

\begin{lemma}\label{lemma:relationshipB}
For $t$, $n\in\N$ we have
\[
B_t^{\dagger\prime}(n) + B_{t-1}^{\dagger\prime}(n)= \sum_{s=0}^{t-1} \binom{t-1}{s} B^\dagger_{s}(n-1).
\]
\end{lemma}

\begin{proof}
Let $\pi^\dagger_{(n)} = 1_{\BSym_{n-1}}\!\ind^{\BSym_n}$ and
 $\theta^\dagger_{(n)} = \pi^\dagger_{(n)} - 1_{\BSym_n}$.
By Lemma~\ref{lemma:branchingB},
\[ B_t^{\dagger\prime}(n) + B_{t-1}^{\dagger\prime}(n)
= \bigl\langle \bigl( \theta^\dagger_{(n)} \bigr)^{t-1}
\pi^\dagger_{(n)}, 1_{\BSym_n} \bigr\rangle. \]
Recall from \S4.2 that %$\pi^\dagger = 1_{\BSym_{n-1}}\ind^{\BSym_n}$ and that
\[ \pi^\dagger_{(n)} = %1_{\BSym_{n-1}}\Ind^{\BSym_n} =
\chi^{((n) \sep \varnothing)} +  \chi^{( (n-1,1) \sep \varnothing)}
 + \chi^{( (n-1) \sep (1) )}. \]
By the Branching Rule for $\BSym_n$ stated in Lemma \ref{lemma:branchingB} we have
%\begin{align*}
\[
\theta^\dagger_{(n)} \Res_{\BSym_{n-1}} =
 2\chi^{((n-1) \sep \varnothing )} + \chi^{( (n-2,1) \sep \varnothing )} + \chi^{((n-2) \sep (1))} \\
 %&
 = \pi^\dagger_{(n-1)} + 1_{\BSym_{n-1}}.
 \]
 %\end{align*}
%This `extra' trivial character accounts for the binomial coefficients below.
%
%Omitting the analogues of the more routine steps above we have
A straightforward calculation using Lemma~\ref{lemma:branchingB} now shows that
\begin{align*}
B_t^{\dagger\prime}(n) + B_{t-1}^{\dagger\prime}(n) &= \bigl\langle
(\theta^\dagger_{(n)})^{t-1} \pi_{(n)}, 1_{\BSym_n} \bigr\rangle \\
&= \bigl\langle \bigl( \theta^\dagger_{(n)} \Res_{\BSym_{n-1}}\bigr)^{t-1}
\Ind^{\BSym_n}, 1_{\BSym_n} \bigr\rangle \\
&= \bigl\langle \bigl( \pi^\dagger_{(n-1)}
+ 1_{\BSym_{n-1}} \bigr)^{t-1} \Ind^{\BSym_n}, 1_{\BSym_n} \bigr\rangle \\
&= \bigl\langle \bigl( \pi^\dagger_{(n-1)} + 1_{\BSym_{n-1}} \bigr)^{t-1}, 1_{\BSym_{n-1}}
\bigr\rangle \\
&= \sum_{s=0}^{t-1} \binom{t-1}{s} \bigl\langle (\pi^\dagger_{(n-1)})^s, 1_{\BSym_{n-1}}
\bigr\rangle,
\end{align*}
and the result follows.
\end{proof}

In order to state a Type B analogue of Proposition~\ref{prop:star}, we define
a further family of Type B Stirling numbers by  $\stir{t}{n}^{\dagger\star}
= \stir{t}{n}^{\dagger\prime} + \stir{t-1}{n}^{\dagger\prime}$ for $t \in \N$ and $n \in \N_0$.

\begin{proposition}\label{prop:starB}
For $t \in \N$ and $n\in \N_0$,  $\stir{t}{n}^{\dagger\star}$ is equal to the number of marked set partitions of $\{1,\ldots, t\}$ into $n$ parts, such that an even number of elements of each part are marked, and such that if $i$ and $i+1$ are in the same part then $i+1$ is marked.
\end{proposition}

\begin{proof}
%By definition $\stir{t}{n}^{\dagger '}$ is the number of marked set partitions of $\{1,\ldots,t\}$ into
%exactly $n$ parts such that if $1$ and $t$ are in the same part then $1$ is marked, and if both
%$i$ and $i+1$ are in the same part then $i+1$ is marked.
By definition of $\stir{t}{n}^{\dagger\prime}$, the marked set partitions of $\{1,\ldots,t\}$ into
exactly $n$ parts counted by $\stir{t}{n}^{\dagger\star}$ but not by $\stir{t}{n}^{\dagger\prime}$
are those with $1$ and $t$ in the same part, but with $1$ unmarked. The following procedure defines a  bijection between these partitions and those counted by $\stir{t-1}{n}^{\dagger\prime}$: if $t$ is marked then transfer the mark to $1$; then delete~$t$.
\end{proof}

Continuing the analogy with Type A, we now use
Lemma~\ref{lemma:relationshipB} to get the generating function for $B^{\dagger\prime}_t(t)$.

\begin{proposition}\label{prop:BdaggerprimedGF}
\[ \sum_{t=0}^\infty \frac{B^{\dagger\prime}_t(t)}{t!}x^t
= \exp \bigl( \mfrac{1}{2}(\exp 2x-1-2x) \bigr).\]
\end{proposition}

\begin{proof}
Taking $t\in \N$ and $n \ge t$,
Lemma~\ref{lemma:relationshipB} gives $B^{\dagger\prime}_t(t) + B^{\dagger\prime}_{t-1}(t-1)
= \sum_{s=0}^{t-1} B^{\dagger\prime}_s(s)
\binom{t-1}{s}$. Observe that if $G(x) = \sum_{t=0}^\infty a_t x^t/t!$
then $x^kG(x)/k! = \sum_{t=0}^\infty a_t x^{t+k}\binom{t+k}{k}/(t+k)!$.
It follows that if
$H(x) = \sum_{t=0}^\infty B^{\dagger\prime}_t(t)x^t/t!$ then
\[ H(x) + H'(x) = \sum_{k=0}^\infty \frac{x^k}{k!} \sum_{t=0}^\infty \frac{B^\dagger_t(t)}{t!} x^t
= \exp \bigl(x+ \mfrac{1}{2} (\exp 2x-1)  \bigr). \]
Since $ \exp \bigl( \mfrac{1}{2}(\exp 2x-1-2x) \bigr)$ solves this differential equation,
and agrees with $H$ when $x=0$, we have $H(x) =\exp \bigl( \mfrac{1}{2}(\exp 2x-1-2x) \bigr)$.
\end{proof}

Thus $\sum_{t=0}^\infty B^{\dagger\prime}_t(t) x^t/t! = \exp(-x) \sum_{t=0}^\infty B^\dagger_t(t) x^t/t!$.
As an immediate corollary we get $B^{\dagger\prime}_t(t) = \sum_{s=0}^t \binom{t}{s} (-1)^{t-s}
B^\dagger_s(s)$
and $B^\dagger_t(t) = \sum_{s=0}^t \binom{t}{s} B^{\dagger\prime}_s(s)$;
again these formulae are related by binomial inversion.
Using this formula for $B^{\dagger\prime}_t(t)$
and~\eqref{eq:BdaggerExp},
we obtain
%\begin{equation} \label{eq:BExp} B_t(t) = \frac{1}{\mathrm{e}} \sum_{j=1}^\infty
%\frac{j^t}{j!} \end{equation}
%(see for example \cite[Equation (1.41)]{Wilf}, or, alternatively,
%sum equation~\eqref{eq:stirling} below over all $n \in \N$)
%we obtain %NB this is 1.6.10 in the second edition
\begin{equation}
\label{eq:BdaggerprimedExp}
B^{\dagger\prime}_t(t) = %\frac{(-1)^t}{\sqrt{e}} +
\frac{1}{\sqrt{e}}
\sum_{j=0}^\infty \frac{(2j-1)^t}{2^j j!}.
\end{equation}
This equation has the same relationship to~\eqref{eq:BdaggerExp} as~\eqref{eq:BprimedExp}
has to~\eqref{eq:BExp}, and the same method used to prove Corollary~\ref{cor:Bernhart} gives
\[ B^{\dagger\prime}_t(t) = 2^{t-1/2} \e^{\ell(t)-t}\ell(t)^{t-1/2} \sqrt{\frac{\ell(t)}{\ell(t)+t}} \bigl( 1 + \O( t^{-1/2} \log t) \bigr)
\]
and
\[ \frac{B^{\dagger\prime}_t(t)}{B^\dagger_t(t)} \sim \frac{1}{\sqrt{2\ell(t)}}
\text{ as $t \rightarrow \infty$} \]
where $\ell(t)$ is as defined in \S6.3.

We also obtain
the expected analogue of Corollary~\ref{cor:Bernhart2}. The example of $\stir{t}{1}^{\dagger\prime}$
shows that, as before, this corollary does not extend to the corresponding Stirling numbers.

\begin{corollary}
If $t\in\N_0$ then $B^{\dagger\prime}_t(t)$ is the number of
set partitions of $\{1,\ldots,t\}$ into parts
of size at least two with an even number of their
elements marked.\hfill$\qed$
\end{corollary}

\begin{proof}
As in the Type A case, this follows from  \cite[Theorem~3.11]{Wilf}
since $\mfrac{1}{2}\bigl( \exp(2x)-1-2x \bigr)$ is the exponential
generating function enumerating sets of size at least two with an even
number of their elements marked.
\end{proof}

%Again we note that this corollary does not extend to the Stirling numbers $\stir{t}{n}^{\dagger\prime}$.
%For example, if $t \ge 2$ then $\stir{t}{1}' = 0$, whereas the unique
%set partition of $\{1,\ldots, t\}$ into a single part obviously has all parts of size at least $2$.

\subsection{Relating $B^\dagger_t(n)$ to $B^{\dagger\prime}_t(n)$ and the asymptotics of $\stir{t}{n}^{\dagger}$, $\stir{t}{n}^{\dagger\prime}$,
$B'_t(n)$ and $B^{\dagger\prime}_t(n)$} % and $B^{\dagger '}_t(t)$}

%This turned out to be wrong (since the differential equation is tractable!)
%---------------------------
%While Lemma~\ref{lemma:relationshipB} gives an efficient way to compute the numbers $B^{\dagger '}_t(n)$,
%it does not appear to give an easy route to the generating function for $B^{\dagger\prime}_t(t)$.
%
%%using the generating function for $B^\dagger_t(n)$ below, %OMIT, have many ways to compute B^\dagger_t(n)
%Proposition~\ref{prop:Bernhart} does not appear to have an analogue for Type B.
%For fixed~$n$ however, the methods used in Type A can easily be applied, and these
%give an alternative route to the analogue of~\eqref{eq:BdaggerExp}, and hence to the
%asymptotics of~$B^{\dagger '}_t(t)$.

Let $t$, $n \in \N$.
Since $\stir{t}{n}^\dagger = 2^{t-n}\stir{t}{n}$, Equation~\eqref{eq:stirling} implies that
\begin{equation}
\label{eq:StBformula}
\stir{t}{n}^\dagger =  \frac{1}{2^nn!}\sum_{k=0}^{n} (-1)^k\binom{n}{k} \bigl(2 (n-k)\bigr)^t.
\end{equation}
 Moreover,
the ordinary generating function for $\stir{t}{n}^\dagger$ is
\begin{equation}
\label{eq:StBgf}
\sum_{t=0}^\infty \stir{t}{n}^\dagger x^t = x^n \prod_{j=1}^{n} \frac{1}{1-2jx}
%x^n\prod_{j=1}^{n} \frac{1}{1-2jx} %\\
\end{equation}
and so $\stir{t}{n}^\dagger$ and $B^\dagger_t(n)$ are both asymptotically equal to
$(2n)^t/2^nn!$ as $t \rightarrow \infty$.
It follows from Lemma~\ref{lemma:relationshipB} that $\stir{t}{n}^{\dagger\prime}+
\stir{t-1}{n}^{\dagger\prime} = \sum_{s=0}^{t-1} \binom{t-1}{s} \stir{s}{n-1}^\dagger$
provided $n \ge 2$. Hence
\begin{equation}
\label{eq:StBdaggergf}
\sum_{t=0}^\infty \stir{t}{n}^{\dagger\prime} x^t = \frac{x^n}{1+x} \prod_{j=1}^{n} \frac{1}{1-(2j-1)x}
\end{equation}
provided $n \ge 2$. A residue calculation then shows that
\[
\stir{t}{n}^{\dagger\prime} \sim \frac{(2n-1)^{t}}{2^nn!} \quad\text{as $t \rightarrow \infty$}. %\\
\]
provided $n \ge 2$. (When $n=1$ we have $\stir{t}{1}^{\dagger\prime} = 1$ if $t$ is even and $\stir{t}{0}^{\dagger\prime} = 0$ if $t$ is odd.)
It easily follows that if $n \ge 2$ then
 $\stir{t}{n}^{\dagger\prime}$ and $B^{\dagger\prime}_t(n)$ are both asymptotically equal to
$(2n-1)^{t-1}/2^n n!$ as $t\rightarrow \infty$.
Moreover, by calculating all residues we obtain the explicit formula
\begin{equation}
\label{eq:StBdaggerformula} \stir{t}{n}^{\dagger\prime}
= \frac{1}{2^n n!}\sum_{k=0}^{n-1} (-1)^k\binom{n}{k} \bigl( 2(n-k)-1 \bigr)^t
+ \frac{(-1)^{n+t}}{2^n n!}
\end{equation}
valid for $n \ge 2$.
This formula has a striking similarity to~\eqref{eq:stirlingprime}.
We remark that an alternative derivation of~\eqref{eq:BdaggerprimedExp} is given by summing
this formula over all $n \in \N_0$, using the identity
\[ \sum_{n=0}^\infty \frac{(-1)^n}{2^nn!}\binom{n}{j} = \frac{1}{ \sqrt{e}}\frac{(-1)^j}{2^j j!}, \]
which can easily be proved by comparing coefficients of $x^j$ in
\[ \sum_{j=0}^\infty \sum_{n=0}^\infty \frac{(-1)^n}{2^nn!}\binom{n}{j}x^j
= \sum_{n=0}^\infty \frac{(-1)^n}{2^nn!} (1+x)^n = \exp \bigl( -\mfrac{1}{2}(1+x)
\bigr) = \frac{\exp(-\mfrac{1}{2}x)}{\sqrt{e}}.
\]
The recurrences for $\stir{t}{n}^\dagger$ and $\stir{t}{n}^{\dagger\prime}$ obtained from~\eqref{eq:StBgf}
and~\eqref{eq:StBdaggergf} are
$\stir{t}{n}^\dagger = \stir{t-1}{n-1}^\dagger + 2n \stir{t-1}{n-1}^\dagger$ for $n \ge 1$
and
$\stir{t}{n}^{\dagger\prime} = \stir{t-1}{n-1}^{\dagger\prime} + (2n-1)\stir{t-1}{n}^{\dagger\prime}$
for $n \ge 3$.

\subsection{Stirling numbers in Type D}
It is possible to define Stirling numbers for Type D, by
$\stir{t}{n}^\ddagger = \MD_t(n) - \MD_t(n-1)$ and
$\stir{t}{n}^{\ddagger\prime} = \MDp_t(n) -\MDp_t(n-1)$.
However, just as the set partitions corresponding to Type D shuffles have no apparent natural description, these Type D Stirling numbers (which take negative values in many cases) have no obvious combinatorial interpretation. We shall say nothing more about these quantities here. We
note however that there is an analogue
of Lemmas~\ref{lemma:Bernhart} and~\ref{lemma:relationshipB} in Type D, namely
\[M_t^{\ddagger\prime}(n) + M_{t-1}^{\ddagger\prime}(n)
= \sum_{s=0}^{t-1} \binom{t-1}{s} B^{\ddagger}_t(n) \]
for $t$, $n \in \N$. This has a similar proof to Lemma~\ref{lemma:relationshipB},
replacing Lemma~\ref{lemma:branchingB}
with  Lemma~\ref{lemma:branchingD}. We also note that $S^\dagger_t(n) = S^{\ddagger}_t(n)$
if $t < n$, since no sequence of $t$ oriented random-to-top shuffles leaving the deck
fixed can lift the bottom card, and evenly many card flips occur in any such sequence.
Moreover $S^\dagger_t(t) + 1 = S^{\ddagger}_t(t)$; the exceptional sequence
lifts the bottom card to the top $t$ times.

\section{Earlier work on shuffles and Kronecker powers, and an obstruction to an explicit bijection between the sequences counted by $B_t(n)$ and $M_t(n)$}

The connection between the random-to-top shuffle (or its inverse, the top-to-random shuffle)
and Solomon's descent algebra is well known. See for example
the remark on Corollary 5.1 in \cite{DFP}.
%Theorem 4.1 in \cite{DFP} gives
%another proof of  Phatarfod's result  using explicit formula for the probability of obtaining
%a particular permutation after $t$ shuffles.

In \cite{GoupilChauve} the authors study the powers of the irreducible character $\chi^{(n-1,1)}$.
%Besides the special case of Lemma~\ref{lemma:branching} already mentioned, the main result
%relevant to this paper is their Lemma 2.
Suppose that $t \le n$. % and note that
%$M'_t(n) = M'_t(n)$.
The special case $\lambda = \varnothing$ of Lemma 2 of \cite{GoupilChauve}
gives a bijection between the sequences of partition moves counted by $M'_t(n)$ and
permutations $\pi$ of $\{1,\ldots, t\}$ such that every cycle of $\pi$ is decreasing,
and $\pi$ has no fixed points. Such permutations are in bijection with set
partitions of $\{1,\ldots,t\}$ into non-singleton parts, and so, using the result of
Bernhart mentioned earlier, we obtain a bijective proof that $M'_t(t) = M'_t(n) = B'_t(n) = B'_t(t)$.
The restriction $t \le n$, which allows the use of the Bernhart result,
arises from the use of the RSK correspondence for oscillating tableaux,
and appears essential to the proof in \cite{GoupilChauve}.

In \cite{FulmanCardShuffling} Fulman defines, for each finite group $G$ and each subgroup $H$
of~$G$,
a Markov chain on the irreducible representations
of $G$. In the special case when $G = \Sym_n$ and $H = \Sym_{n-1}$
this gives a Markov chain $J$ on the set of partitions of $n$ in
which the transition probability from the partition $\lambda$ to the partition $\mu$ is
\[ \frac{M_1(\lambda,\mu)}{n} \frac{\chi^\mu(1)}{\chi^\lambda(1)} . \]
%is non-zero if and only if there is a move from $\lambda$ to $\mu$, in which case
%the probability is $\frac{1}{n} \chi^\mu(1) / \chi^\lambda(1)$.
%His main theorem uses the RSK correspondence (see \cite[CHK]{StanleyII}): let
%$\sh(\alpha) = \lambda$ if under the RSK correspondence the permutation $\alpha$
%is sent to a pair of tableaux of shape $\lambda$.
The relevant special case of his Theorem~3.1 is as follows.

\begin{proposition}[Fulman] %, \protect{\cite[Theorem 3.1]{Fulman}}]
Let $\lambda$ be a partition of $n$.
Starting at $(n)$, the probability that $J$ is at $\lambda$ after $t$ steps is equal to the
probability that a product %$\alpha$
of $t$ random-to-top shuffles %satisfies $\sh(\alpha)
%= \lambda$.
is sent to a pair
of tableaux of shape $\lambda$ by the RSK correspondence.
\end{proposition}

%Since $\sh(\alpha) = \sh(\alpha^{-1})$ for any permutation $\alpha$,
%Since the RSK shapes of a permutation and its inverse are the same
%(see \cite[CHK]{StanleyII}), Fulman's theorem
%also holds for the random-to-top shuffle.
Fulman's theorem relates partitions
and shuffle sequences, but it is qualitatively different to Lemma~\ref{lemma:SM}.
In particular, we note that Lemma~\ref{lemma:SM} has no very natural probabilistic interpretation,
since choosing individual moves uniformly at random fails to give a uniform distribution on the
set of sequences of all partition moves from $(n)$ to $(n)$.

Despite this, it is very natural to ask whether the RSK correspondence can be used
to give a bijective proof of Lemma~\ref{lemma:SM}, particularly in view of the related
RSK correspondence used in \cite{GoupilChauve}, and the fact
that if $\tau \in \Sym_n$ then the RSK shape of $\tau \sigma_m$
differs from the RSK shape of $\tau$ by a move. (We leave this remark
as a straightforward exercise.)
The following proposition
shows that, perhaps surprisingly, the answer to this question is negative.
The proof is a brute-force verification, which to make this paper self-contained
we present in Figure~1 above.

\begin{proposition}\label{prop:noseqs}
Let $\sh(\tau)$ denote the RSK shape of the permutation $\tau$ and
let
\[ (\lambda^{(0)}, \ldots, \lambda^{(8)}) =
\bigl( (5), (4,1), (3,2), (4,1), (3,2), (2,2,1), (3,2), (4,1), (5) \bigr). \]
There does not exist a sequence $(\tau_1, \ldots, \tau_8)$ of random-to-top shuffles such
that $\sh(\tau_1\ldots \tau_i) = \lambda^{(i)}$ for all $i \in \{0,1,\ldots,8\}$.\hfill$\qed$
 %$\tau_1 \ldots \tau_8 = \id$ and
%\[ \begin{split}
%\bigl( \sh(\id), \sh(\tau_1), {}&{} \sh(\tau_1\tau_2), \ldots, \sh(\tau_1\tau_2 \ldots \tau_8) \bigr)
%\\ &= \bigl( (5), (4,1), (3,2), (4,1), (3,2), (2,2,1), (3,2), (4,1), (5) \bigr).
%\end{split}
%\]
\end{proposition}

%\begin{proof}
%This is checked in Figure~1 below.
%\end{proof}

%\enlargethispage{3pt}
\begin{figure}[t]
\includegraphics[width=\textwidth]{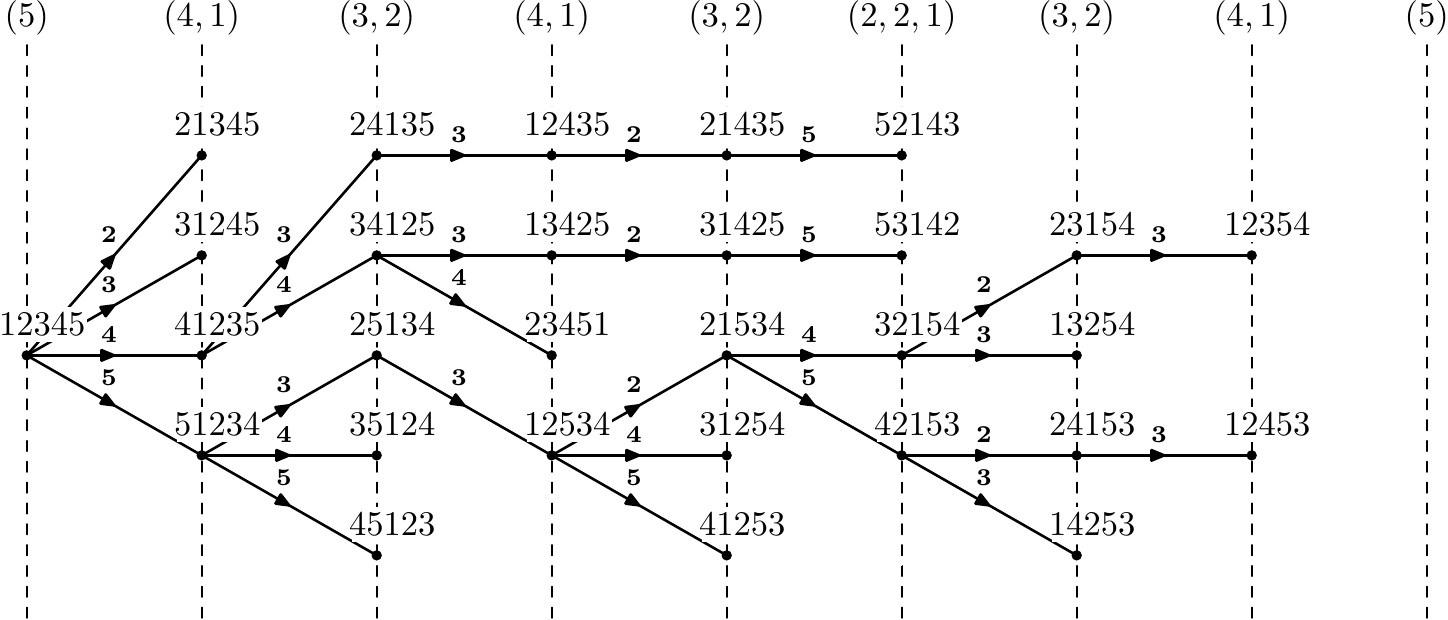}
\caption{All sequences $(\tau_1, \tau_2, \ldots, \tau_r)$ of random-to-top shuffles such
that $\sh(\tau_1,\ldots \tau_i) = \lambda^{(i)}$ for each $i \in \{0,1,\ldots, r\}$,
where $\lambda^{(1)}, \ldots, \lambda^{(8)}$ are as
defined in Proposition~\ref{prop:noseqs}. Labels on vertices show the
inverses of the permutations $\tau_1 \ldots \tau_i$ and so give the order of the deck;
an arrow labelled $m$ corresponds to the random-to-top shuffle $\sigma_m$.}
\end{figure}

Computer calculations show that if $t \le 7$ then taking shapes of RSK correspondents
gives a bijection proving Theorem~\ref{thm:main} whenever $n \le 12$ and that when $t =8$,
similar examples  to the one above
exists for $5 \le n \le 8$. Since a shuffle sequence counted by $S_t(n)$ can never move a
card that starts in position $t+1$ or lower, taking shapes of RSK correspondents
fails to give a bijective proof of Theorem~\ref{thm:main} whenever $t=8$ and $n \ge 5$.
It is therefore an open problem to find a bijective proof
of Theorem~\ref{thm:main} that deals with the case $t > n$ not covered by
Lemma 2 of \cite{GoupilChauve}. Finding bijective
proofs of Theorems~\ref{thm:mainB} and~\ref{thm:mainD} are also open problems.

\section{Occurrences of Type A and B Bell and Stirling numbers in the On-Line Encyclopedia of Integer Sequences}

The numbers $B'_t(t)$ appear in the On-Line Encyclopedia of Integer Sequences
(OEIS) \cite{OEIS} as sequence A000296; the interpretations
at the start of \S 6.1 and in Corollary~\ref{cor:Bernhart2} are both given. The sequences $B'_t(n)$ and
$\stir{t}{n}'$ as $t$ varies
appear for $n \le 3$, and the sequence $\stir{t}{4}'$ appears as A243869,
counting the number of set partitions of $\{1,\ldots,t\}$ into four parts, satisfying our condition.

The Type B Stirling numbers $\stir{t}{n}^\dagger$ appear
as sequence A075497; they are defined there by their characterization as ordinary Stirling numbers scaled by particular powers of $2$. They do not appear to have been connected before with Type~B Coxeter groups or descent algebras.

The other statistics introduced here for types B and D have not hitherto appeared in OEIS in any generality, and appear not to have been defined prior to this investigation. They appear in OEIS only for certain very particular choices of parameters, usually those for which the numbers have particularly simple expressions. For instance, we see from (\ref{eq:StBformula}) that $\stir{t}{1}^\dagger = 2^{t-1}$, and that $\stir{t}{2}^\dagger=2^{t-2}(2^{t-1}+1)$, giving sequences A000079 and A007582 respectively .

The sequence $B^\dagger_t(3)$ also appears, as sequence A233162. Let $C$ be a set of~$2n$ unlabelled colours arranged into $n$ pairs. Let $A$ be an $s\times t$ array of boxes. Define $Q_{s,t}(n)$ to be the number of ways of colouring $A$ with colours of $C$, so that no two horizontally or vertically adjacent boxes receive colours from the same pair. (Since the colours are unlabelled,
two colourings of $A$ that differ by a permutation of the $n$ pairs,
or by swapping colours within a pair, are regarded as the same.)
The sequence A233162 gives the numbers $Q_{1,t}(4)$.
%special case that $n=4$ and $s=1$.
However the coincidence with the Type B Bell numbers extends to any number of colours, as follows.
\begin{proposition}
$Q_{1,t}(n)=B^\dagger_{t-1}(n-1)$ for all $t\ge 1$.
\end{proposition}
\begin{proof}

Let the boxes be labelled $0,\dots, t-1$ from left to right. We shall show that the number of colourings using exactly $n$ colour pairs (though not necessarily using both colours in each pair) is
%the Type B Stirling number
$\stir{t-1}{n-1}^\dagger$; this is the number of marked  partitions of a set of size $t-1$ into exactly $n-1$ parts. Since $B^\dagger_{t-1}(n-1)=\sum_{k=0}^{n-1}\stir{t-1}{k}^\dagger$,
this is enough to prove the proposition.

Suppose we are given a colouring of the boxes using exactly $n$ colour pairs, such that consecutive boxes receive colours from distinct pairs. We label the colour pairs with the numbers $1,\dots, n$ in order of their first appearance in the colouring, reading from left to right. Let $c_i$ be the label of the colour of box $i$. We define a function $f$ on $\{1,\dots,t-1\}$ by
\[
f(i) = \left\{\begin{array}{ll} c_i & \textrm{\ if\ } c_i<c_{i-1},\\ c_i-1 & \textrm{\ if\ } c_i>c_{i-1}.\end{array}\right.
\]
(By assumption, the case $c_i=c_{i-1}$ does not occur.) We notice that if the first occurrence of the colour pair $k$ occurs in box $i$, where $i>0$, then $c_j<k$ for all $j<i$. It follows that $f(i)=k-1$, and hence that the image of $f$ is $\{1,\dots, n-1\}$. Now the kernel congruence of $f$ defines a set partition of $\{1,\dots,t-1\}$ into exactly $n-1$ parts; it remains to determine the marks.

Within each colour pair, we now distinguish a marked and an unmarked colour: the marked colour is the one which first appears in our colouring. (It is not necessary that the unmarked colour should appear at all.) Let $X$ be a part of the partition just described, and let $x$ be the least element of~$X$. We add marks to the elements of $X$ as follows: for each element $y\in X$ other than $x$, we add a mark to $y$ if box $y$ receives a marked colour. Then we add a mark to $x$ if necessary to make the total number of marks in $X$ even.

Given a marked partition of $\{1,\dots,t-1\}$ into exactly $n-1$ parts, it is easy to reverse the procedure just described in order to reconstruct the unique colouring of $t$ boxes with $n$ unlabelled colour pairs with which it is associated. So we have a bijective correspondence, and the proposition is proved.
\end{proof}

%A007582

%%Note acm is also good.
%%Hacked version of amsplain is in /Users/uvah099/Library/texmf/bibtex/bst/amscls
%\bibliographystyle{amsplain-no-dash}
%\bibliography{../References/References}

\def\cprime{$'$} \def\Dbar{\leavevmode\lower.6ex\hbox to 0pt{\hskip-.23ex
  \accent"16\hss}D} \def\cprime{$'$}
\providecommand{\bysame}{\leavevmode\hbox to3em{\hrulefill}\thinspace}
\providecommand{\MR}{\relax\ifhmode\unskip\space\fi MR }
\providecommand{\MRhref}[2]{%
  \href{http://www.ams.org/mathscinet-getitem?mr=#1}{#2}
}
\providecommand{\href}[2]{#2}

\end{document}